\documentclass{article}
\usepackage{amsmath}
\usepackage[utf8]{inputenc}
\usepackage{amssymb}
\usepackage{amsfonts}
\usepackage{graphicx}
\usepackage{ulem}
\usepackage[table,dvipsnames]{xcolor}
\usepackage{epstopdf}
\usepackage{amsthm} 
\usepackage{thmtools}
\usepackage{color}
\usepackage[colorlinks]{hyperref}
\usepackage{cleveref}
\usepackage[comma,numbers,square,sort&compress]{natbib} 
\usepackage{geometry} 
\usepackage{enumitem} 
\usepackage{stmaryrd} 
\usepackage{authblk}

\usepackage{tikz}
\usepackage{pgfplots}
\usetikzlibrary{arrows.meta,positioning,decorations.markings} 
\definecolor{grey}{cmyk}{0,0,0,0.5}

\newcommand{\RR}{\mathcal{R}} 
\newcommand{\R}{\mathbb{R}}
\newcommand{\N}{\mathbb{N}}
 
\newcommand{\supp}{\text{supp}}
\newcommand{\A}{\mathcal{A}} 
\newcommand{\V}{\mathcal{V}}
\newcommand{\W}{\mathcal{W}}
\renewcommand{\L}{\mathcal{L}}
\newcommand{\X}{\mathcal{X}}
\newcommand{\B}{\mathcal{B}}
\newcommand{\Co}{\mathcal{C}}
\newcommand{\OO}{\mathcal{O}}
\renewcommand{\d}{\mathrm{d}} 
\newcommand{\n}{2} 
\newcommand{\ep}{\varepsilon}

\newtheorem{theorem}{Theorem}[section]
\newtheorem{proposition}[theorem]{Proposition}
\newtheorem{definition}[theorem]{Definition}
\newtheorem{assumption}[theorem]{Assumption}
\newtheorem{corollary}[theorem]{Corollary}
\newtheorem{remark}[theorem]{Remark}
\newtheorem{lemma}[theorem]{Lemma}
\newtheorem{exemple}[theorem]{Example}

\begin{document}

\title{Competitive exclusion and coexistence in a model with age-structured resource
}
\date{}
\author[1]{S.~Girel}
\affil[1]{Université Côte d’Azur, CNRS, LJAD, Parc Valrose, 06108 Nice, France}
		
\author[2]{Q.~Richard}
\affil[2]{Université de Montpellier, CNRS, IMAG, 34090 Montpellier, France.}

\maketitle

\hrule

\begin{abstract}
In this paper, we introduce and analyse a model in which a shared, age-structured resource is exploited by two competing populations through age-dependent interaction kernels. It provides a general framework for studying resource-mediated competition in epidemiology, immunology and ecology. 

We use integrated semigroups theory to establish the well-posedness of the model and its fundamental properties, including positivity and the existence of a global attractor. Then, by means of Lyapunov functionals, persistence results and generalised LaSalle principle, we fully characterize the asymptotic dynamics in terms of the  competitors's basic reproduction numbers and invasive fitnesses.   

Our results demonstrate that in addition to the classical competitive exclusion, the system can exhibit globally asymptotically stable coexistence, an outcome allowed by a mechanism of age-mediated niche partitioning.
\end{abstract}


\vspace{0.3cm}

\hrule

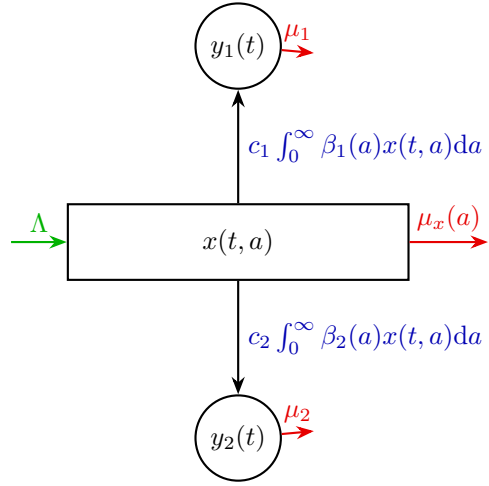
\begin{figure} 
\begin{center}
\begin{tikzpicture}[>=Stealth, node distance=2cm, thick, main/.style = {draw, circle}]
\tikzstyle{fleche}=[->,>=Stealth,thick,rounded corners=4pt]

    \node[main] (A) [draw, rectangle, minimum width=4.5cm, minimum height=1cm] {$x(t,a)$};
    \node[main] (B) [above = 1.5cm of A] {$y_1(t)$};
    \node[main] (C) [below = 1.5cm of A] {$y_2(t)$};

    \draw[fleche] (A) -- node[midway, right] {$\textcolor{black!30!blue}{c_1\int_0^\infty \beta_1(a)x(t,a)\d a}$}(B);
    \draw[fleche] (A) -- node[midway, right] {$\textcolor{black!30!blue}{c_2\int_0^\infty \beta_2(a)x(t,a)\d a}$}(C);
  
    \draw[fleche,black!30!green] (-3,0) -- node[midway,above] {$\Lambda$} (A);
    \draw[fleche,black!10!red] (A) -- node[midway,above] {$\mu_x(a)$}(3.3,0);
    \draw[fleche,black!10!red] (B) -- node[midway,above] {$\mu_1$}(1,2.5);
    \draw[fleche,black!10!red] (C) -- node[midway,above] {$\mu_2$}(1,-2.5);
  
 \end{tikzpicture}
\end{center}\caption{Graphical representation of model \eqref{Eq:model}}
\end{figure}

\section{Introduction}

Understanding the outcome of competition for a shared resource is a recurrent question in biological fields such as ecology, epidemiology or immunology. In this context, the competitive exclusion principle, or Gause's principle, states that only the strongest competitor, typically characterized by the largest basic reproduction number, survives while all others go extinct. 

This principle faces limitations, as coexistence is frequently observed in real biological systems and is often explained by niche partitioning mechanisms, whereby competitors adopt different strategies to exploit the resource. These include spatial or temporal separation of predation, or specialization on distinct subsets of the resource \cite{Murdoch2003,Chesson2000, Kotepui2020}. These observations have motivated the development of numerous mathematical models to identify which mechanisms may relax competitive exclusion. Many of them are extensions of the  Kermack–McKendrick epidemiological model or of the competitive Lotka-Volterra equations \cite[Sections 3.5 and 10]{Murray89} incorporating various sources of variability such as interaction terms that do not obey the law of mass action  \cite{Castillo2020,Castillo2022,Chesson2000}, spatial heterogeneity through diffusion and/or advection \cite{Qin2022,Tuncer2012, Ackleh2016}, time delay \cite{Freedman1989} or an effect of seasonality \cite{Andreasen2023,Martcheva2009}.

In this paper, we investigate the effect of age-structure of the resource population, motivated by the idea that specialization of competitors for different ages of the resource might allow for competitors coexistence \cite{Roos2008}. We consider the following model, in which two populations, $y_1$ and $y_2$, compete for the same age-structured resource $x$:

\begin{equation}
\label{Eq:model}
\left\{
\begin{array}{rcl}
\partial_t x(t,a)+\partial_a x(t,a)&=&-\mu_{x}(a) x(t,a)-x(t,a)(\beta_1(a)y_1(t)+\beta_2(a)y_2(t)),\\
y_1'(t)&=&y_1(t)c_1\int_0^\infty \beta_1(a)x(t,a)da-\mu_{1} y_1(t),\\
y_2'(t)&=&y_2(t)c_2\int_0^\infty \beta_2(a)x(t,a)da-\mu_{2} y_2(t),\\
x(t,0)&=&\Lambda, \\
(x(0,\cdot),y_1(0),y_2(0))&=& (x_0,y_{1,0},y_{2,0})\in L^1(\R_+,\R_+)\times \R_+^2.
\end{array}
\right.
\end{equation}

Death rates are described by the parameters $\mu_i$. The functions $\beta_i$ are age-dependent interaction kernels, allowing each competitive population to exploit different segments of the structured resource. This heterogeneity captures biological situations where the resources differ in accessibility and suitability across their life cycle. In ecological systems, this corresponds for instance to predators specializing on juvenile versus adult prey, or plants versus fruits; in epidemiology or immunology, it reflects the spread of two strains among susceptible individuals or immune cells, where the infection probability is age or stage dependent. 

Unlike epidemiological models, where interactions induce a direct transfer from the susceptible to the infected class, ecological predation involves losses that are not fully converted into predator growth. This is embedded in the conversion factors $c_1$ and $c_2$, as in the Lotka-Volterra equations, which are set to $1$ in epidemiological context.

An assumption of the model is that the influx of resources of age 0 is constant, \textit{i.e.} $x(t,0)=\Lambda$, rather than linear in the density of resources. This assumption is standard in mathematical immunology and epidemiology \cite{martcheva2015}. While less common in ecology, it is relevant in situations involving allochthonous resources \cite{Nevai2012,Jansen2018}, \textit{i.e.} resources that originate from a place other than where they will live, such as algae dragged by a river to the shore, seeds spread by the wind or by animals. It also makes sense when the resource is a product of human activity, such as breeding, plantation, or industrial byproduct. It is also appropriate for competitive species in a chemostat \cite{Freedman1989}, where the controlled resource supply is constant.

The problem of competition for the same age-structured resource with a linear source term has, to our knowledge, not been investigated. Only the case of a single (unstructured) predator hunting an age-structured prey has been analysed in \cite{RichardPerasso2019} (see also the references therein), then generalising the predator-prey Lotka-Volterra equations to age-structured prey. The authors show that not only this system can produce periodic solution (as in the Lotka-Volterra ODE system), but also extinction of both population, convergence to a positive equilibrium or unbounded solutions.

In \cite{Djidjou2022b}, Djidjou-Demasse et al. proposed a within-host competition model in which a population of red blood cells (the resource) can be infected by multiple \textit{Plasmodium} strains. The resource is discretely structured into three stages, while the parasite populations are structured by a continuous age variable. Moreover, infection is mediated by the release of merozoites and thus deviates from the mass-action assumption.  Although this modeling framework is tailored to a specific biological system, it addresses similar questions to those we are considering here. In particular, the authors show that stable coexistence may occur under explicit conditions on the parameters, a feature that we also recover in our more general setting.

Models with structures closely related to \eqref{Eq:model} have been considered in which the competitors are age-structured while the resource is not and with $c_i=1$.  First Magal, McCluskey and Webb \cite{MagCluskWebb2010} analyzed the single-strain case, later extended to multiple strains by Martcheva and Li  \cite{MartchevaLi2013}, and further refined in \cite{Richard2020}, for the two strains case, in which the degenerate situation where two strains sharing equal reproduction numbers is considered. In a recent work \cite{GirelRichard26}, we completed this line of research by treating the case of an arbitrary number of strains where the maximal reproduction number can be shared by several strains. These works consistently show that age-structure in the competitors alone does not allow for robust coexistence: competitive exclusion holds generically while coexistence can only occur for singular parameter values, resulting in a continuum of equilibria.

Instead, in this paper we show that considering age-structure in the resource has a fundamentally different effect, as it allows for structurally stable coexistence  under suitable conditions on the kernels $\beta_i$. In particular, a necessary condition is that $\beta_2$ and $\beta_1$ are not proportional,  which is consistent with the concept of  niche partitioning.

This paper is organized as follows. We first establish the well-posedness of the model using integrated semi-group theory and the method of characteristics. We characterize the  omega-limit sets of each initial condition and prove the existence of an asymptotically stable global attractor. Then we state the conditions of existence of equilibria in terms of basic and invasion reproduction numbers.

Next, we define candidate Lyapunov functionals. Those functions are not defined on the whole phase space. We prove that, thanks to the regularization induced by the boundary condition, they are defined and differentiable along all complete orbits included in suitable invariant subsets, and conclude that they are valid Lyapunov functionals on those subsets. In particular, the proof that those functions are well defined on the omega-limit sets requires to demonstrate the uniform persistence of the flow,  using criteria from \cite{SmithThieme2011,HaleWaltman89} based on the behaviour of the flow near the equilibria included in the boundary.

Finally, we compute the basin of attraction of each equilibrium and prove its stability. This is achieved by considering that, under suitable conditions,  \textit{i)} the omega-limit set of each compact set is attractive, \textit{ii)} a Lyapunov functional is defined on this set and \textit{iii)} this omega-limit  set contains an equilibrium. Using a generalisation of the LaSalle invariance principle \cite[Theorem 2.52]{SmithThieme2011}, we conclude that each omega-limit set actually reduces to a single equilibrium.

\section{Framework and preliminary results}

We begin this section by introducing the framework and the different notations. Then we establish several properties of the semiflow, such as its well-posedness, positivity, dissipativity as well as the existence of a global attractor. Finally we state the fundamental properties of the omega-limit sets.

\subsection{Assumptions and notations}
Throughout this article, we consider that the following assumption is satisfied.
\begin{assumption}\label{Assum:param}
We suppose that:
\begin{itemize}
    \item $\mu_x, \beta_1$ and $\beta_2$ are in $L^\infty(\R_+,\R_+)$,
    \item $\Lambda, \mu_1$ and $\mu_2$ are in $\R_+^*$,
    \item there exists $\mu_0>0$ such that $\mu_x(a)\geq \mu_0$ a.e. $a\geq 0$.
\end{itemize}
\end{assumption}

We introduce the notations $\underline{\mu}=\min\{\mu_0,\mu_1,\mu_2\}$, $\overline{\mu}=\max\{\|\mu_x\|_{L^\infty},\mu_1,\mu_2\}$, $\underline{c}=\min\{1,c_1,c_2\}$ and $\overline{c}=\max\{1,c_1,c_2\}$. 

Let us define, $X_+=L^1(\R_+,\R_+)\times \R_+^2$ and, for $i\in\{1,2\}$, the sets $X_i=\{(x,y_1,y_2)\in X_+: y_i>0\}$ and $\partial X_i = \{(x,y_1,y_2)\in X_+: y_i=0\}$.

\begin{definition}
Let $\Phi:\R_+\times X_+\to X_+$ be a semiflow. A set $S_1\subset X_+$ is said to attract a set $S_2\subset X_+$ if $S_1\neq \emptyset$ and $$d(\Phi_t(S_2),S_1):= \underset{s_2\in S_2}{\sup} \underset{s_1\in S_1}{\inf}  ||\Phi_t(s_2)-s_1||_X \underset{t\to +\infty}{\to}0.$$ 
\end{definition}

\subsection{Qualitative properties of the semiflow}
We first use the theory of integrated semigroups to assess the existence of a unique mild (integrated) solution to system \eqref{Eq:model} for any initial condition $(x_0,y_{1,0},y_{2,0})\in X_+$. We use the notations and definitions from  \cite{MagalRuan2018,SmithThieme2011} and introduce the linear operator $A:D(A)\subset \X \to \X$ by 
$$A\begin{pmatrix}
\begin{pmatrix}
0 \\
x
\end{pmatrix}\\
y_1\\
y_2
\end{pmatrix}=\begin{pmatrix}
\begin{pmatrix}
-x(0)\\
-x'-\mu_x x
\end{pmatrix} \\
-\mu_1 y_1\\
-\mu_2 y_2
\end{pmatrix}$$
where $D(A)=(\{0\}\times W^{1,1}(\R_+,\R))\times \R^2$ and $\X=(\R\times L^1(\R_+,\R))\times \R^2$. We also define $\X_0=\overline{D(A)}=(\{0\}\times L^1(\R_+,\R))\times \R^2$ and its positive cone $\X_{0+}=(\{0\}\times L^1(\R_+,\R_+))\times \R_+^2$. We introduce $A_0$, the part of $A$ on $\X_0$, \textit{i.e.} the restriction of $A$ to $$D(A_0)=\{x\in D(A)\ :\ Ax\in\overline{D(A)}\}.$$
Finally, we define the non-linear function $F:\X_0\to \X$  by 
$$F\begin{pmatrix}
\begin{pmatrix}
0 \\
x
\end{pmatrix}\\
y_1\\
y_2
\end{pmatrix}
=
\begin{pmatrix}
\begin{pmatrix}
\Lambda \\
-y_1\beta_1 x - y_2\beta_2 x
\end{pmatrix} \\
c_1 y \int_0^\infty \beta_1(a)x(a) \d a\\[0.1cm]
c_2 z \int_0^\infty \beta_2(a)x(a) \d a
\end{pmatrix}.$$
Then system \eqref{Eq:model} can be rewritten as the following abstract semilinear Cauchy problem with $u(t)=((0,x(t,\cdot)),y_1(t),y_2(t))$
\begin{equation}
\label{Eq:modelCauchy}
\forall t\geq 0,\quad u'(t)=Au(t)+F(u(t)) \ ;\ u(0)=u_0\in \X_{0+}.
\end{equation}
Since $\X_0=\overline{D(A)}\subsetneq \X$, \textit{i.e.} $A$ is a non-densely defined operator, integrated solutions of \eqref{Eq:modelCauchy} will belong to $\X_0$. Moreover, the non-linear operator $A+F$ satisfies the following properties. 

\begin{lemma}\mbox{}  
\label{Lemma:PropCauchy}
\begin{enumerate}
    \item $A$ is a Hille-Yosida operator,  with $(-\underline{\mu},+\infty)\subset \rho(A)$, where $\rho(A)$ denotes the resolvent set of $A$, $A$ is resolvent positive (\textit{i.e.} $\forall \lambda\in\rho(A),\ (\lambda I-A)^{-1}\X_+\subset \X_+)$ and
    $$\|(\lambda I-A)^{-n})\|_{\L(\X)}\leq \dfrac{1}{(\lambda+\underline{\mu})^n}, \ \forall \lambda>-\underline{\mu}, \ \forall n\geq 1.$$
    \item $F$ is a locally Lipschitz continuous function: $\forall r>0$, $\exists K_r>0$ such that 
    $$\|F(u_0)-F(\tilde{u}_0)\|_{\X}\leq K_r\|u_0-\tilde{u}_0\|_{\X}$$
    for every $(u_0,\tilde{u}_0)\in (\X_0\cap \B_{\X}(0,r))^2$ where $\B_{\X}(0,r)$ is the ball of $\X$ centered at $0_\X$ with radius $r$.\end{enumerate} 
\end{lemma}

\begin{proof}
\begin{enumerate}
\item Let $\lambda\in \R$, $((0,x),y_1,y_2)\in D(A)$ and $(\alpha,\varphi,c,d)\in \X$. Then $$(\lambda I_{D(A)}-A)
\begin{pmatrix}
\begin{pmatrix}
0 \\
x
\end{pmatrix}\\
y_1\\
y_2
\end{pmatrix}
=
\begin{pmatrix}
\begin{pmatrix}
\alpha \\
\varphi
\end{pmatrix}\\
c\\
d
\end{pmatrix}
\Leftrightarrow \left\{ \begin{array}{rcl}
x(0) &=& \alpha\\
x' &=&-(\lambda+\mu_x)x+\varphi \\
y_1&=&\frac{c}{\lambda+\mu_1}\\[0.1cm]
y_2&=&\frac{d}{\lambda+\mu_2}
\end{array}  \right.$$
where the  first two equations are equivalent to $x(a)=\alpha e^{-\int_0^a \lambda+\mu_x(s)\d s}+\int_0^a \varphi(s)e^{-\int_s^a\lambda + \mu_x(r) \d r}\d s$ for a.e. $a\geq 0$. 

Then for any $\lambda>\underline{\mu}$, $(\lambda I_{D(A)}-A)$ is a bijection from  $D(A)$ to $\X$ and  $A$ is resolvent positive. Moreover, the previous expression of $x$ leads, for any $\lambda>\underline{\mu}$ and any $V=  (\alpha,\varphi,c,d)\in \X$, to $$||(\lambda I_{D(A)} -A)^{-1}V||_\X  \leq \frac{\alpha + ||\varphi||_{L^1}}{\lambda+\underline{\mu}} +\frac{c}{\lambda+\mu_1}+\frac{d}{\lambda+\mu_2}\leq \frac{||V||_\X}{\lambda+\underline{\mu}}$$

and then, as the operator norm is sub-multiplicative, for any $n\geq 1$, to $$||(\lambda I_{D(A)} -A)^{-n}||_{\L(\X)}\leq ||(\lambda I_{D(A)} -A)^{-1}||^n_{\L(X)} =\underset{V\in\X, ||V||_\X=1}{\sup}||(\lambda I_{D(A)} -A)^{-1}V||^n_X\leq \frac{1}{(\lambda+\underline{\mu})^n}.$$

\item Standard computations show that $F$ is locally Lipschitz continuous, where the local Lipschitz constant can be chosen as $K_r=r\sum_{i=1}^2 ||\beta_i||_{L^\infty}(1+c_i)$. 
\end{enumerate}
\end{proof}

\begin{proposition}\label{Prop:solutions}
\begin{enumerate} 
    \item For any $u_0\in\X_{0+}$, there exists a unique non-negative maximal integrated solution $u$ to \eqref{Eq:modelCauchy}, \textit{i.e.} $u\in \Co([0,t_{max}(u_0)),\X_{0+}) $ such that 
    $$\int_0^t u(s)\in D(A), \quad \forall t\in[0,t_{\max}(u_0))$$
    and
    \begin{equation}
        u(t)=u_0+A\int_0^t u(s)\d s+\int_0^t F(u(s))\d s, \quad \forall t\in[0,t_{\max}(u_0)). \label{Eq:mild}
    \end{equation}
    \item This solution is global in time, \textit{i.e.} $t_{\max}(u_0)=+\infty$ and it defines a continuous semiflow $U:\R_+\times \X_{0+}\to \X_{0+}$, $(t,u_0)\longmapsto u(t)$. Moreover, for any $t\geq 0$ and any $u_0\in \X_{0+}$, the total mass $||u(t)||_{\X}$ of  system \eqref{Eq:model}  satisfies, for all $t\geq 0$:
    \begin{equation}\label{Eq:MassMajoration}
        ||u(t)||_{\X}\leq \overline{c}||u_0||_\X e^{-\underline{\mu}t} +  \frac{\overline{c}\Lambda}{\underline{\mu}}(1-e^{-\underline{\mu}t})  \leq \overline{c}\max\left(||u_0||_{\X},\frac{\overline{c}\Lambda}{\underline{\mu}}\right).
    \end{equation}

    \item The trivial bijection from $\X_{0+}$ to $X_+$, obtained by removing the first coordinate, defines a contiuous semiflow for \eqref{Eq:model}  
    $$\Phi:\R_+\times X_+ \to X_+,\quad v=(x_0,y_{1,0},y_{2,0}) \longmapsto  \Phi_t(v)=(\Phi_t^x,\Phi_t^{1},\Phi_t^{2})(v)=(x(t,.),y_1(t),y_2(t)).$$  
    
    For all $t\geq  0$ and $v=(x_0,y_{1,0},y_{2,0})\in X_+$, the semiflow satisfies $\Phi_t^x(v)(a)=\widetilde{\Phi_t^x}(v)(a)+\widehat{\Phi_t^x}(v)(a)$ with, a.e. $a\geq 0$,  
    \begin{equation}\label{Eq:Duhamel}
        \left\{\begin{array}{rcl} \widetilde{\Phi_t^x}(v)(a)&=&x_0(a-t)e^{-\int_{a-t}^a \left( \mu_x(s)+\sum_{i=1}^2\beta_i(s)\Phi_{s-a+t}^{i}(v)\right)\d s}\chi_{[t,\infty)}(a),\\
        \widehat{\Phi_t^x}(v)(a)&=&\Lambda e^{-\int_{0}^a              \left(\mu_x(s)+\sum_{i=1}^2\beta_i(s)\Phi_{t-a+s}^{i}(v)\right)\d s}\chi_{[0,t)}(a).
        \end{array}
        \right.
    \end{equation}
 
    \item The semiflow $\Phi$ is bounded-dissipative on $X_+$, \textit{i.e.} there exists a bounded subset that attracts each bounded subset of $X_+$.
    \item The semiflow $\Phi$ is asymptotically smooth, \textit{i.e.} for every non-empty, bounded and positively invariant set $B\subset X_+$, there exists a compact set $K\subset B$ such that $K$ attracts $B$. 
    \item The semiflow $\Phi$ is state-continuous, uniformly in finite time, \textit{i.e.} for every $(t,v)\in \R_+\times X_+$ and every $\ep >0$, there exists $\delta>0$ such that $$\forall (s,\tilde{v})\in[0,t]\times X_+ : \|v-\tilde{v}\|_{X}\leq \delta \Longrightarrow \|\Phi_s(v)-\Phi_s(\tilde{v})\|_{X}\leq \ep .$$
\end{enumerate}
\end{proposition}

\begin{proof}
\begin{enumerate}
    \item From Lemma \ref{Lemma:PropCauchy}, $A$ is a Hille-Yosida operator, so it generates a locally Lipschitz continuous integrated semigroup $\{S_A(t)\}_{t\geq 0}$ on $\X$ (\cite[Theorem 2.4]{KellermanHieber89} or \cite[Proposition 3.4.3, p.116]{MagalRuan2018}). Moreover, Kellermann-Hieber theorem (\cite{KellermanHieber89} or \cite[Theorem 3.6.2, p.133]{MagalRuan2018}) states that for any $\tau>0$ and any $f\in L^1((0,\tau),\X)$, the maps $t\longmapsto \left(S_A \ast f\right)(t)$ are continuously differentiable on $(0,\tau)$, where $\ast$ denotes the convolution product:
    $$(S_A \ast f)(t)=\int_0^t S_A(t-s)f(s)ds$$
    and for all $t\in[0,\tau]$
    \begin{equation}\label{Eq:convol-estim}
    \|u_f(t)\|:=\left\|\dfrac{d}{dt}(S_A \ast f)(t)\right\|\leq  \int_0^t e^{-\underline{\mu} (t-s)}\|f(s)\| \d s.  
    \end{equation}
    If in particular $f\in \Co((0,\tau),\X)$, then $\|u_f(t)\|\leq \delta_A(t)\underset{s\in[0,t]}{\sup}||f(s)||$ where $\delta_A(t)=t$ converges to $0$ as $t\to 0$ and Assumption 5.1.2 from \cite{MagalRuan2018} applies. Since we proved that $F$ is locally Lipschitz continuous, \cite[Theorem 5.2.7, p.226]{MagalRuan2018} states that there exists a unique maximal integrated  solution $U(.)u_0\in\Co([0,t_{\max}),\X_0)$ (with $t_{\max}\leq +\infty$) to  problem \eqref{Eq:modelCauchy} for each initial condition $u_0\in \X_0$. This solution satisfies:
    \begin{equation}\label{Eq:u_convol}
        U(t)u_0=T_{A_0}(t)u_0+\dfrac{d}{dt}\left(S_A \ast (F\circ U(.)u_0)(t)\right)
    \end{equation}
    where $\{T_{A_0}\}_{t\geq 0}$ is the $\Co_0$-semigroup on $\X_0$ generated by $A_0$. The existence of $\{T_{A_0}\}_{t\geq 0}$ on $\overline{D(A_0)}$ comes from the Hille-Yosida Theorem \cite[Lemma 2.4.5]{MagalRuan2018} while Lemma 2.2.11 from \cite{MagalRuan2018} implies that $\overline{D(A_0)}= \X_0$.

    For any $\gamma\leq 0$, the operator $A-\gamma I$ is resolvent positive, as a direct consequence of $A$ beeing resolvent positive. Moreover, for any $r>0$ and any $u_0\in \X_{0+}$ such that $\|u_0\|_\X \leq r$, $F(u_0)+\gamma_r u_0\in \X_{+}$, where $\gamma_r=2r(\|\beta_1\|_{L^\infty}+\|\beta_2\|_{L^\infty})$. Then 
    \cite[Theorem 5.3.2]{MagalRuan2018} states that the integrated solution is non-negative.

    \item From \eqref{Eq:mild}, the mild solution of \eqref{Eq:modelCauchy} is such that, for any $u_0=((0,x_0),y_{1,0},y_{2,0})\in\X_{0+}$ and any $t\in[0,t_{\max})$,
    \begin{flalign}
    x(t,.)&=x_0(.) -\int_0^t \partial_a x(s,.) \d s - \int_0^t \mu_x(.)x(s,.)\d s - \sum_{i=1}^2 y_i(s)\int_0^t \beta_i(.)x(s,.)\d s \label{Eq:IntegratedExpression_x}\\
    y_i(t)&=y_{i,0}-\mu_i\int_0^t y_i(s)\d s +c_i\int_0^t y_i(s)\int_0^\infty \beta_i(a)x(s,a)\d a \d s, \quad \forall i\in\{1,2\}.\label{Eq:IntegratedExpression_y}
    \end{flalign}
    
    Set $N(t)=\overline{c}\int_0^\infty x(t,a) \d a +\sum_{i=1}^2 y_i(t)$, then $N(0)=\overline{c}||x_0||_{L^1}+\sum_{i=1}^2 y_{i,0}$. Using the integrated expressions \eqref{Eq:IntegratedExpression_x}-\eqref{Eq:IntegratedExpression_y}:
    
    \begin{flalign*}
    N(t)=& N(0) + \overline{c}\Lambda t  - \int_0^t \overline{c}\int_0^\infty \mu_x(a) x(s,a) \d a \d s  - \sum_{i=1}^2\left( \mu_i \int_0^t y_i(s) \d s   +    (c_i-\overline{c})\int_0^t y_i(s)\int_0^\infty \beta_i(a)x(s,a)\d a\d s\right)
    \end{flalign*}    
     where we used the following equation since $\int_0^t x(s,.)\d s\in W^{1,1}(\R_+,\R)$
     $$\int_0^t\int_0^\infty \partial_a x(s,a) \d a \d s=-\int_0^t \Lambda \d s = -\Lambda t.$$ 
    Then
    $$N'(t)=\overline{c}\Lambda -  \overline{c}\int_0^\infty \mu_x(a) x(t,a) \d a - \sum_{i=1}^2 \left(\mu_i y_i(t)  +  (c_i-\overline{c}) \int_0^\infty y_i(t) \beta_i(a)\d a\right).$$
    By using the definition of $\underline{\mu}$ and truncating the $(c_i-\overline{c})$ negative contribution we obtain $N'(t)\leq \overline{c}\Lambda - \underline{\mu} N(t)$. It follows that
    $$||u(t)||_\X=\int_0^\infty x(t,a) \d a +\sum_{i=1}^2 y_i(t) \leq N(t)\leq N(0)e^{-\underline{\mu}t}+\frac{\overline{c}\Lambda }{\underline{\mu}}(1-e^{-\underline{\mu}t})$$
    by using Gronwall's lemma which leads to \eqref{Eq:MassMajoration} since $N(0)\leq \overline{c}||u_0||_\X$. In particular, $t\longmapsto ||u(t)||_\X$ is bounded on $[0,t_{\max}(u_0))$, then  $t_{\max} = +\infty$.

    \item Formula \eqref{Eq:Duhamel} are derived from the method of characteristics (or Duhamel's formula).
    \item From \eqref{Eq:MassMajoration}, for any bounded set $B\subset X_+$, $\underset{t\to +\infty}{\lim}||\Phi_t(v)||_\X \leq \frac{\overline{c}\Lambda}{\underline{\mu}}$ uniformly in $v\in B$, then the closed ball of $X_+$ centered at $0_{X}$ with radius $\frac{\overline{c}\Lambda}{\underline{\mu}}$ attracts all bounded sets of $X_+$. 
    \item We prove the asymptotic smoothness of $\Phi$ using Lemma 3.2.6 from \cite{Hale1988} (see also \cite[Theorem 2.46]{SmithThieme2011}). For any $t\in \R_+$, the semiflow $\Phi_t$ decomposes as $\Phi_t=(\widetilde{\Phi_t^x},0,0)+(\widehat{\Phi_t^x},\Phi_t^{1},\Phi_t^{2})$. 
    
    First, equation \eqref{Eq:Duhamel} implies that the function $\delta:R_+\times\R_+\ni (t,r) \longmapsto re^{-\mu_0t} \in\R_+$ is such that, for any $r>0$, $t\geq 0$ and $v\in X_+$ such that $||v||_X\leq r$ then $||(\widetilde{\Phi_t^x}(v),0,0)||_{X}\leq \delta(t,r)$ with $\underset{s\to +\infty}{\lim}\delta(s,r)=0$. 
    
    Next we prove that, for any fixed $t\geq 0$, $(\widehat{\Phi_t^x},\Phi_t^{1},\Phi_t^{2})$ is a compact operator on any bounded subset of $X_+$, that is for any bounded set $B\subset X_+$, $(\widehat{\Phi_t^x},\Phi_t^{1},\Phi_t^{2})(B)$ is a precompact set of $X_+$, \textit{i.e.} $\widehat{\Phi_t^x}(B)$ is a precompact set of $L^1(\R_+,\R_+)$. 
 
    It remains to prove the three assumptions of the Fréchet-Kolmogorov theorem \cite[Theorem X.1, p. 275]{Yosida80}. Let $r>0$ and $B\subset \B_{X_+}(0,r)$, the ball of $X_+$ centered at $0_{X}$ with radius $r$.
    \begin{description}
    \item[i)] The boundedness of $\widehat{\Phi_t^x}(B)$ is given by equation \eqref{Eq:MassMajoration}, which readily implies that for any $v\in B$ $$||\widehat{\Phi_t^x}(v)||_{L^1}\leq ||\Phi_t(v)||_{X_+}\leq \overline{c} \max\left(r,\frac{\Lambda}{\underline{\mu}}\right)=:M_r.$$
    
    \item[ii)] For any $v\in B$, the equitightness property $\underset{d\to +\infty}{\lim}\int_{a>d}\widehat{\Phi_t^x}(v)(a)\d a=0$ uniformly on $\widehat{\Phi_t^x}(B)$ is also readily satisfied because $\widehat{\Phi_t^x}(v)(a)=0$ for $a>t$ (independently of $v\in B$).
    
    \item[iii)] It remains to prove that for any $t\geq 0$, $\lim_{h\to 0} \int_\R |\widehat{\Phi_t^x}(v)(a+h)-\widehat{\Phi_t^x}(v)(a)|\d a=0$, uniformly in $v\in B$. 
    
    From \eqref{Eq:Duhamel}, for any $v\in B$ and any $t,a\geq 0$, $\widehat{\Phi_t^x}(v)(a)=\Lambda e^{A_{t,v}(a)}\chi_{[0,t)}(a)$, where 
    $$A_{t,v}:a\longmapsto-\int_{0}^a \left( \mu_x(s)+\sum_{i=1}^2\beta_i(s)\Phi_{t-a+s}^{i}(v)\right) \d s.$$
    If $t=0$, the  integral is readily zero. Let $v\in B$, $t> 0$ and $h\in(0,t]$. Since $\widehat{\Phi_t^x}(v)(a+h)=0$ if $a>t-h$ and $\widehat{\Phi_t^x}(v)(a)=0$ if $a>t$, then
    \begin{flalign*}
    &\int_\R  |\widehat{\Phi_t^x}(v)(a+h)-\widehat{\Phi_t^x}(v)(a)|\d a \\
    &= \int_{-h}^0 \widehat{\Phi_t^x}(v)(a+h)da + \int_0^{t-h} |\widehat{\Phi_t^x}(v)(a+h)-\widehat{\Phi_t^x}(v)(a)|\d a+ \int_{t-h}^t |\widehat{\Phi_t^x}(v)(a)|\d a
    \end{flalign*}
    The first and third terms are clearly bounded by $\Lambda h$. The second term is controlled by using the mean value inequality $|e^{-a}-e^{-b}|\leq e^{-\min(a,b)} |a-b|\leq |a-b|$ for any $a,b>0$:
    \begin{flalign*}
    &|\widehat{\Phi_t^x}(v)(a+h)-\widehat{\Phi_t^x}(v)(a)|\leq \Lambda |A_{t,v}(a+h)-A_{t,v}(a)|\\
    &\leq \Lambda \int_a^{a+h} \left(\mu_x(s)+\sum_{i=1}^2\beta_i(s)\Phi^{i}_{t-a+s}(v)\right)\d s +  \Lambda \sum_{i=1}^2\int_0^a \beta_i(s)\left|\Phi^{i}_{t-a-h+s}(v)-\Phi^{i}_{t-a+s}(v)\right|   \d s.
    \end{flalign*}
    From \eqref{Eq:MassMajoration} and the boundedness of $B$, there exist constants $C_0,C_1,C_2>0$ such that   $$\int_a^{a+h}\left(\mu_x(s)+\sum_{i=1}^2\beta_i(s)\Phi^{i}_{t-a+s}(v)\right)\d s \leq C_1 h $$
    and, using the integrated expression  \eqref{Eq:IntegratedExpression_y},
    \begin{flalign*}
    &\beta_i(s)|\Phi^{i}_{t-a-h+s}(v)-\Phi^{i}_{t-a+s}(v)|=\beta_i(s)\left|c_i\int_{t-a-h+s}^{t-a+s}\Phi^{i}_{\sigma}(v)\left( \int_0^\infty \beta_i(\alpha)\Phi^x_{\sigma}(v)(\alpha)\d \alpha -\mu_i\right)\d \sigma\right|\\
    &\leq ||\beta_i||_{L^\infty} c_i M_r (||\beta_i||_{L^\infty} M_r +\mu_i)h\leq C_i  h.
    \end{flalign*}
    
    Finally, $$\int_0^{t-h} |\widehat{\Phi_t^x}(v)(a+h)-\widehat{\Phi_t^x}(v)(a)|\d a\leq (t-h)\Lambda(\sum_{i=0}^2 C_i)h$$

    and then    
    $$ \int_0^\infty |\widehat{\Phi_t^x}(v)(a+h)-\widehat{\Phi_t^x}(v)(a)|\d a\leq \Lambda\left(2+(t-h)(\sum_{i=0}^2C_i)\right)h$$
    converges to $0$ as $h\to 0$, uniformly on $v\in B$.
    \end{description}

    Consequently, the Fréchet-Kolmogorov theorem states that $\widehat{\Phi_t^x}(B)$ is a precompact set of $L^1(\R_+,\R_+)$, then the set $(\widehat{\Phi_t^x}, \Phi_t^{1},\Phi_t^{2})(B)$ is precompact in $X_+$ and this achieves the proof.
    \item Let $u_0\in \X_{0+}$, $t\geq 0$, $\ep >0$ and set
    $$r=\overline{c}\max\left(\dfrac{\Lambda}{\underline{\mu}},\|u_0\|_{\X}+\ep \right), \qquad \delta=\ep  e^{-K_r t}$$ 
    where $K_r$ is the Lipschitz constant from Lemma \ref{Lemma:PropCauchy}. Let $\tilde{u}_0\in \X_{0+}$ such that $\|u_0-\tilde{u}_0\|_{\X}\leq \delta$. 
    Since $U(.)u_0$, $U(.)\tilde{u}_0$ and $F$ are continuous, then $F\circ U(.)u_0-F\circ U(.)\tilde{u}_0$ belong to $\in \Co(\R_+,\X)\subset L^1([0,t],\X)$. Moreover, since $\|\tilde{u}_0\|_{\X}\leq \delta+\|u_0\|_{\X}\leq r$, we know from \eqref{Eq:MassMajoration} that for every $s\in[0,t]$, $(U(s)u_0,U(s)\tilde{u}_0)\in (\X_{0+}\cap B_{\X}(0,r))^2$. From \eqref{Eq:u_convol}, it comes for any $s\in[0,t)$:
    $$U(s)\tilde{u}_0-U(s)u_0=T_{A_0}(s)(\tilde{u}_0-u_0)+\dfrac{d}{dt}\left(S_A \ast ((F\circ U(.)\tilde{u}_0)-(F\circ U(.)u_0))(s)\right).$$
    Then the Hille-Yosida Theorem applied to $A_0$ (notice that $\rho(A_0)=\rho(A)$ by \cite[Lemma 2.2.10]{MagalRuan2018}) states that 
    $$\|T_{A_0}(s)(\tilde{u}_0-u_0)\|_{\X}\leq e^{-\underline{\mu} s}\|\tilde{u}_0-u_0\|_{\X}\leq \delta, \ \forall s\in[0,t].$$
    Finally, equation \eqref{Eq:convol-estim} with $f=F\circ U(.)u_0-F\circ U(.)\tilde{u}_0$ rewrites as
    $$\|U(s)\tilde{u}_0-U(s)u_0\|_{\X}\leq \delta+\int_0^s \|F(U(\xi)\tilde{u}_0)-F(U(\xi)u_0)\|_{\X}d\xi \leq \delta+K_r\int_0^s \|U(\xi)\tilde{u}_0-U(\xi)u_0\|_{\X}d\xi, \quad \forall s\in[0,t]$$
    which leads by means of Gronwall's inequality to
    \begin{equation*}
        \|U(s)\tilde{u}_0-U(s)u_0\|_{\X}\leq \delta e^{K_r s} \leq \ep , \quad \forall s\in [0,t],
    \end{equation*}
    which is the expected bound, where $v$ and $\tilde{v}$ are obtained from $u_0$ and $\tilde{u}_0$ by removing the first coordinate.
\end{enumerate}
\end{proof}
 
\begin{corollary}There exists a unique compact attractor of bounded set $\A\subset X_+$, that is $\A$ is a non-empty, compact and invariant subset that attracts every bounded subsets of $X_+$. This attractor satisfies $\A\subset B_{X_+}(0,\frac{\overline{c}\Lambda}{\mu_0})$. Moreover, the set $\A$ is locally asymptotically stable (LAS), \textit{i.e.} for every neighborhood $\V$ of $\A$, there exists a neighborhood $\W\subset \V$ of $\A$ such that $\Phi_t(\W)\subset \V$ for
every $t\geq 0$.
\label{Cor:attractor}  \end{corollary}

\begin{proof}
The existence of a unique compact attractor $\A$ is given by \cite[Theorem 2.33]{SmithThieme2011} (or \cite[Theorem 3.4.6]{Hale1988}) as a consequence of $\Phi$ being asymptotically smooth and bounded-dissipative. Since $\Phi$ is also state-continuous, uniformly in finite time, \cite[Theorem 2.39]{SmithThieme2011} states that $\A$ is locally asymptotically stable.
\end{proof}
For each $v\in X_+$, we denote by $\OO_v=\{\Phi_t(v),t\geq 0\}$ the positive orbit starting from $v$ and 
$$\omega(v)=\bigcap_{\tau \geq 0}\overline{\{\Phi_t(v),t\geq \tau\}}$$
the $\omega$-limit set of $v$. 

\begin{lemma}[Properties of omega-limit set]\label{Lemma:omega-limit}
For every $v\in X_+$,
\begin{itemize}
    \item $\omega(v)$ is non-empty, compact and connected,
    \item $\omega(v)$ is invariant under $\Phi_t$,
    \item $\omega(v)$ attracts $v$, \textit{i.e.} $\lim_{t\to\infty}d(\Phi_t(v),\omega(v))=0$.
\end{itemize}
\end{lemma}
\begin{proof}
Under the hypotheses already checked to prove the asymptotic smoothness of the semiflow, \cite[Proposition 3.13, p.100]{Webb85} states that for every $v\in X_+$, the orbit $\OO_v\subset X_+$ is relatively compact ($\overline{\OO_v}$ is compact). Then \cite[Lemma 3.1.1 and 3.1.2, p. 36]{Hale1988} (or  \cite[Thm 4.1, p. 167]{Walker1980}) give the conclusions of the Lemma.
\end{proof}

\section{Main results on asymptotic dynamics}

In this section, we identify the equilibria of \eqref{Eq:model} in $X_+$. In particular, we interpret the conditions for existence of positive coexistence equilibria (i.e. in $X_1\cap X_2$) from a biological point of view. Next, we state the main results of this paper, which give en exhaustive characterisation of the attractivity and stability of equilibria, their proofs are deferred to Section \ref{Sec:GlobalAnalysis}.

\subsection{Existence of equilibria}
\label{Sec:equilibria}
\begin{proposition}[Equilibria in the boundary $\partial X_1 \cup \partial X_2$]\label{prop:boundary_equilibria}
\ 
\begin{enumerate}
    \item There exists a trivial equilibrium $E_0$, with no competitor, defined by
$$E_0=(x^*_0,0,0)=(\Lambda e^{-\int_0^\cdot \mu_x(s)ds},0,0).$$
\item When they exist, the equilibria with a single competitor are denoted by $E_1=(x^*_1,y^*_{1},0)\in X_1\cap \partial X_2$ and $E_2=(x^*_2,0,y^*_{2})\in X_2\cap \partial X_1$. For each $i,j\in\{1,2\}$ with $i\neq j$, $E_i$ exists if and only if $\RR_{0,i}>1$, where $\RR_{0,i}$ is the  basic reproduction number
\begin{equation}\label{Eq:DefR0i} \RR_{0,i}=\frac{\Lambda c_i}{\mu_i}\int_0^\infty \beta_i(a) e^{-\int_0^a \mu_x(s)ds}da=\frac{c_i}{\mu_i}\int_0^\infty \beta_i(a) x_0^*(a)da,\quad \forall i\in\{1,2\}.\end{equation}

In this case we have 
$$x_i^*(a)=\Lambda e^{-\int_0^a \mu_x(s)ds}e^{-y^*_{i}\int_0^a \beta_i(s)ds}$$
where $y_{i}^* > 0$ is implicitly defined by
\begin{equation} \frac{c_i}{\mu_i}\int_0^\infty \beta_i(a) x_i^*(a)\mathrm{d} a=1\label{Eq:Def_y^*}.
\end{equation}
\end{enumerate}
\end{proposition}  

\begin{proof}
Let $i\in\{1,2\}$, then $E_i$ satisfies
$$(x_i^*)'(a)=-\mu_x x_i^*(a)-x_i^*(a)\beta_i(a)y_i^*$$
which integrates as stated in the proposition.
\end{proof}

 The search of coexistence equilibria in $\X_1\cap X_2$, that is equilibria for which all components are positive, is more complicated. Indeed, a point $(\overline{x},\overline{y_1},\overline{y_2})\in X_+$ with $\overline{y_1}>0, \overline{y_2}>0$ is an equilibrium if and only if it satisfies the following system:
\begin{equation}
\left\{
\begin{array}{llll}
T_k(\overline{y_1},\overline{y_2})&=&1,& \forall k\in\{1,2\} \\
\overline{x}(a)&=&\Lambda e^{-\int_0^a \mu_x(s)\d s}e^{-\sum_{i=1}^2 \overline{y}_i\int_0^a \beta_i(s)\d s},& \forall a\geq 0
\end{array}
\right.\label{Eq:CNS_positive_eq}
\end{equation}
where, for each $k\in\{1,2\}$,
\begin{equation*}
T_k(y_1,y_2)=\frac{\Lambda c_k}{\mu_k}\int_0^\infty \beta_k(a)e^{-\int_0^a \mu_x(s)\d s}e^{-\sum_{i=1}^2y_i\int_0^a \beta_i(s)\d s}\d a.
\end{equation*}
Clearly $T_1< \RR_{0,1}$ and $T_2< \RR_{0,2}$ on $(\R_+^*)^2$, then it follows from the first condition in \eqref{Eq:CNS_positive_eq} that there is no positive equilibrium if $\RR_{0,1}\leq 1$ or $\RR_{0,2}\leq 1$. To investigate the possibility of equilibrium if $\RR_{0,1}>1$ and $\RR_{0,2}>1$, observe that $T_1(0,0)=\RR_{0,1}$ and  $T_1(y_{1}^*,0)=1$  while $T_2(0,0)=\RR_{0,2}$ and $T_2(0,y_{2}^*)= 1$. We then define
\begin{equation}
    \begin{array}{l}
        \gamma_{2\to 1}= T_1(0,y_{2}^*)=\frac{\Lambda c_1}{\mu_1}\int_0^\infty \beta_1(a) e^{-\int_0^a \mu_x(s)\d s}e^{-y_{2}^*\int_0^a \beta_2(s)\d s}\d a=\frac{c_1}{\mu_1}\int_0^\infty \beta_1(a) x_2^*(a) \d a,\\
        \gamma_{1\to 2}= T_2(y_{1}^*,0)=\frac{\Lambda c_2}{\mu_2}\int_0^\infty \beta_2(a) e^{-\int_0^a \mu_x(s)\d s}e^{-y_{1}^*\int_0^a \beta_1(s)\d s}\d a=\frac{c_2}{\mu_2}\int_0^\infty \beta_2(a) x_1^*(a) \d a.
    \end{array}\label{Eq:Def_gamma}
\end{equation}

\begin{remark}\label{Remark: Comparison_gamma_R}
Let $(i,j)\in\{1,2\}^2$, $i\neq j$, then $\gamma_{i\to j}$ is defined if and only if $\RR_{0,i}>1$. Moreover, $\gamma_{i\to j} \leq \RR_{0,j}$, with strict inequality if and only if $\|\beta_j\|_{L^1}> 0$. 
\end{remark}

Those quantities are known as \textit{invasion fitnesses} or \textit{invasion reproduction numbers} and can be explained as in adaptive dynamics where we consider that populations are at equilibrium $E_1$ (respectively $E_2$) and we want to know if newly introduced competitor $y_2$ (resp. $y_1$) can invade.

\begin{lemma}\label{Lemma: CNS_eq_inf}
Assume that $\RR_{0,1}>1$ and $\RR_{0,2}>1$, then the following assumptions are equivalent: \begin{enumerate}
    \item $\gamma_{1\to 2}=\gamma_{2\to 1}=1$.
    \item $y_{2}^*\beta_2(a)=y_{1}^*\beta_1(a)$ almost everywhere $a\geq 0$.
    \item $c_2\mu_1 \beta_2(a)=c_1\mu_2\beta_1(a)$ almost everywhere $a\geq 0$.
\end{enumerate}
\end{lemma}
The proof of this Lemma will be handled in Section \ref{SectionProof_Lemma: CNS_eq_inf}.  

\begin{proposition}[Positive equilibria]\label{Prop:Positive_equilibria}
Suppose that $\RR_{0,1}>1$ and $\RR_{0,2}>1$, then \begin{enumerate}
    \item If $\gamma_{1\to 2}>1$ and $\gamma_{2\to 1}>1$ there exists a unique positive equilibrium $E_3=(\overline{x},\overline{y_1},\overline{y_2})$. Moreover $E_3$ is such that $0<\overline{y_1}\leq y_{1}^*$ and $0<\overline{y_2}\leq y_{2}^*$ and $(\overline{y_1},\overline{y_2})\neq (y_1^*,y_2^*)$.

    \item If  $\gamma_{1\to 2}=\gamma_{2\to 1}=1$, then $x_1^*=x_2^*$ and there is a continuum $E_\infty$ of positive equilibria, given by
    $$E_\infty=\{E_{1+\sigma}\ :\ \sigma\in(0,1)\}\text{ where }  E_{1+\sigma}=(1-\sigma) E_1 +\sigma {E_2} =(x_1^*,(1-\sigma) y_1^*, \sigma y_2^*),$$
    i.e. $E_\infty$ is the  segment joining $E_1$ to $E_2$. There is no positive equilibrium outside of $E_\infty$.
    \item Otherwise, there is no positive equilibrium.
\end{enumerate}
\end{proposition}

We will later show (Proposition \ref{Prop:Lyap_welldef}) that the case $\max\{\gamma_{i\to j},\gamma_{j\to i}\}<1$ and the case $\gamma_{i\to j}=1,\ \gamma_{j\to i}<1$ ($i,j\in\{1,2\},\ i\neq j$) never happen. Consequently, the third item actually states that there is no positive equilibrium when  $\gamma_{j\to}>1$ and $\gamma_{i\to j}\leq 1$. This is a direct consequence of the global attractivity of $E_i$ in  $X_1\cap X_2$, which will be stated in Theorem \ref{Thm:AttractivenessR0>1}. 

In this section, we only prove the existence statement of the first item, as well as the second item. The uniqueness of $E_3$ will be proven later (Lemma \ref{Lemma:Lyap_welldef_Orbit}.3). 

\begin{proof}[Proof of Proposition \ref{Prop:Positive_equilibria}]
\begin{enumerate}
    \item Suppose that $\gamma_{1\to 2}>1$ and $\gamma_{2\to 1}>1$. First we note that the functions $\R_+\ni z\longmapsto T_1(z,y_2)$ and $\R_+\ni z\longmapsto T_2(y_1,z)$ are decreasingly going to zero as $z$ goes to $+\infty$. Let us define the two following assumptions:
    \begin{equation}\label{Eq:inf-beta2}
    \inf(\supp(\beta_2))\geq \sup(\supp(\beta_1))
    \end{equation}
    and
    \begin{equation}\label{Eq:inf-beta1}
    \inf(\supp(\beta_1))\geq \sup(\supp(\beta_2)).
    \end{equation}
    We readily see that those two conditions are not compatible since it would lead to $\beta_1=\beta_2\equiv 0$ and then $\RR_{0,1}=\RR_{0,2}=0$ that is absurd.
    
    Suppose now that \eqref{Eq:inf-beta2} holds then the function $\R_+\ni z\longmapsto T_1(y_1,z)$ is constant implying that $T_1(y_1^*,y_2)=1$ for every $y_2\geq 0$. Since \eqref{Eq:inf-beta1} does not hold then the function $\R_+\ni z\longmapsto T_2(z,y_2)$ is decreasingly going to zero as $z$ goes to $+\infty$ for every $y_2\geq 0$. It follows from $T_2(y_1^*,0)=\gamma_{1\to 2}>1$ that $y_1\longmapsto T_2(y_1,0)-1$ has exactly one root $\widehat{y_1}>y_{1}^*$. Consequently, the level curve $\{(y_1,y_2)\in \R_+^2 : T_2(y_1,y_2)=1\}$ continuously joins $(0,y_{2}^*)$ and $(\widehat{y_1},0)$ while $\{(y_1,y_2)\in \R_+^2 : T_1(y_1,y_2)=1\}=\{(y_1^*,y_2), y_2\geq 0\}$. It then readily follows that there exists $(\overline{y_1},\overline{y_2})$ such that $(\overline{x},\overline{y_1},\overline{y_2})$ is an equilibrium, where $\overline{x}$ is defined by \eqref{Eq:CNS_positive_eq} and $\overline{y_1}=y_1^*$. The monotony of $T_2$ also implies that $0<\overline{y_2}< y_{2}^*$ and the equilibrium is unique in this case.
    
    The case where \eqref{Eq:inf-beta1} holds but not \eqref{Eq:inf-beta2} is similar and leads to a unique coexistence equilibrium $(\overline{x},\overline{y_1},y_2^*)$ with $0<\overline{y_1}< y_{1}^*$ by monotony of $T_1$ and $\overline{x}$ is given by \eqref{Eq:CNS_positive_eq}.
    
    Finally, in the case where neither \eqref{Eq:inf-beta1} nor \eqref{Eq:inf-beta2} hold, we see that the functions $y_1\longmapsto T_2(y_1,0)-1$ and $y_2\longmapsto T_1(0,y_2)-1$ have exactly one root $\widehat{y_1}>y_{1}^*$ and $\widehat{y_2}>y_{2}^*$ respectively. Consequently, the level curve $\{(y_1,y_2)\in \R_+^2 : T_1(y_1,y_2)=1\}$ continuously joins $(0,\widehat{y_2})$ and $(y_{1}^*,0)$, while the level curve $\{(y_1,y_2)\in \R_+^2 : T_2(y_1,y_2)=1\}$ continuously joins $(0,y_{2}^*)$ and $(\widehat{y_1},0)$. It then readily follows that there exists  $(\overline{y_1},\overline{y_2})$ such that $(\overline{x},\overline{y_1},\overline{y_2})$ is an equilibrium, where $\overline{x}$ is defined by \eqref{Eq:CNS_positive_eq}. The monotony of $T_1$ and $T_2$ also implies that $0<\overline{y_1}<y_{1}^*$ and $0<\overline{y_2}<y_{2}^*$.
    
    \item If $\gamma_{1\to 2}=\gamma_{2\to 1}=1$, then Lemma \ref{Lemma: CNS_eq_inf} readily implies that $T_1=T_2$ and that for all $\sigma\in (0,1)$, $T_1((1-\sigma )y_1^*,\sigma y_2^*)=1$. Moreover, formula \eqref{Eq:CNS_positive_eq} implies that if there exists $\sigma\in (0,1)$ such that $y_1=(1-\sigma) y_1^*$ and $y_2= \sigma y_2^*$ then $(x^*,y_1^*,y_2^*)$ is an equilibrium if and only if $$x^*(a)=\Lambda e^{-\int_0^a\mu_x(s)\d s} e^{(1-\sigma) y_1^*\int_0^a\beta_1(s) \d s - \sigma y_2^*\int_0^a\beta_2(s) \d s}.$$
    
    Then Lemma \ref{Lemma: CNS_eq_inf} implies that  $x^*=x_1^*=x_2^*$. 
    
    Since $(y_1,y_2) \longmapsto T_1(y_1,y_2)$ is a decreasing function with respect to both $y_1$ and $y_2$ on $\R_+^2$, then $T_1(y_1,y_2)<1$ (resp $>1$) on the right (resp. left) side of the segment joining $(y_1^*,0)$ to $(0,y_2^*)$, so there is no other positive equilibrium.
\end{enumerate}
\end{proof}

\begin{remark}
It is worth noticing that if $\beta_2=c\beta_1$ for some $c\in\R_+^*$, then equations $\eqref{Eq:Def_y^*}$ and \eqref{Eq:Def_gamma} readily imply that $\gamma_{2\to 1}=1/\gamma_{1\to 2}$. Consequently, in this case, either there is a continuum of equilibria (in the critical and structurally unstable case $c=c_1\mu_1/c_2\mu_2$ which, from Lemma \ref{Lemma: CNS_eq_inf}, implies $\gamma_{2\to 1}=\gamma_{1\to 2}=1$), or there is no positive equilibrium and one competitor goes extinct. This is consistent with the belief that niche differentiation is necessary to avoid competitive exclusion. \label{Remark:ProportionalBeta}
\end{remark}

\begin{remark}\label{Remark:OptimalBeta}
Since $x_1^*$ is a decreasing function of $a$, then  the definition of $\gamma_{1\to 2}$ implies that in the situation where the environment is occupied by single population $y_1$ (considered at its equilibrium $E_1$),   the best competitive strategy for a newly introduced population $y_2$ to challenge the resident population $y_1$ (\textit{i.e.} the $\beta_2$ shape that maximizes $\gamma_{1\to 2}$, as we will see with Theorem \ref{Thm:AttractivenessR0>1} that the survival of $y_2$ depends on $\gamma_{1\to 2}$) is such that $\beta_2$ in concentrated near the age $a=0$ (optimally, a Dirac distribution in $0$). Indeed, since the resource influx is constant, then there is no need for the competitors to preserve the resource sustainability (as it would have been in the case of a linear birth rate). Consequently the best strategy is to be the first competitor to collect enough resource.
\end{remark}

\subsection{Attractivity, stability and basins of attraction}
We investigate the stability and attractiveness of every equilibrium, depending on the values of the basic reproduction numbers and invasion fitnesses $R_{0,i}$ and  $\gamma_{i\to j}$ ($i,j\in\{1,2\})$. For the sake of clarity, we gathered hereunder the main results on equilibria attractivity and stability, which are proved in Section \ref{Sec:GlobalAnalysis}. 

\begin{proposition}[Positive invariance of sets]\label{Prop:Attractors}
For $i\in\{1,2\}$, the sets $X_i$ and $\partial  X_i$ are positively invariant, that is, for all $t\geq 0$  $\Phi_t(X_i)\subset  X_i$ and $\Phi_t(\partial  X_i)\subset \partial  X_i$. 
\end{proposition}

\begin{proof}\mbox{}
For any $i\in\{1,2\}$, any initial condition $v=(x_0,y_{1,0},y_{2,0})\in  X_1$ and any $t\geq 0$, $y_i'(t)\geq -\mu_i y_i(t)$, then $y_i(t)\geq y_{i,0}e^{-\mu_i t}>0$. The positive invariance of $\partial X_i$ is straightforward.
\end{proof}

A consequence of Proposition \ref{Prop:Attractors} is that if there exists $i\in\{1,2\}$ such that $v\in \partial X_i$ then the model \eqref{Eq:model} reduces to a single predator model.

\begin{theorem}[Asymptotic stability of $E_0$]  \label{Thm:Attractiveness_E0}Suppose that one of the following assumptions holds:
\begin{enumerate}
        \item $\underset{i\in\{1,2\}}{\max}\{\RR_{0,i}\}\leq 1$ and $S= X_+$
        \item $\RR_{0,1}\leq 1$ and $S=\partial  X_2$
        \item $\RR_{0,2}\leq 1$ and $S=\partial  X_1$
        \item $S=\underset{i\in\{1,2\}}{\bigcap}\partial  X_i$
\end{enumerate}
then $\{E_0\}$  is globally asymptotically stable in $S$.
\end{theorem}

\begin{remark}
More precisely, one can note that the equilibrium $E_0$ is globally exponentially stable in $\bigcap_{i\in\{1,2\}} \partial X_i$. Indeed, consider the initial condition $v=(x_0,0,0)\in \underset{i\in\{1,2\}}{\bigcap} \partial X_i$. Then \eqref{Eq:Duhamel} rewrites as
$$x(t,a)=\begin{cases} x_0(a-t)e^{-\int_{a-t}^a \mu_x(s)ds} &\mbox{if } a\geq t,\\
x_0^*(a) & \mbox{if } t\geq a.
\end{cases}$$
Moreover, for $0\leq t< a$, then $x_0^*(a)=\Lambda e^{-\int_0^a\mu_x(s) \d s}=\Lambda e^{-\int_0^{a-t}\mu_x(s) \d s}e^{-\int_{a-t}^a\mu_x(s) \d s}=x_0^*(a-t)e^{-\int_{a-t}^a\mu_x(s) \d s}.$ We deduce that
$\|x(t,.)-x_0^*\|_{L^1}=\int_t^\infty |(x_0(a-t) -x_0^*(a-t))|e^{-\int_{a-t}^a \mu_x(s)ds}da\leq \|x_0-x_0^*\|_{L^1}e^{-\mu_0 t}\underset{t\to+\infty}{\longrightarrow} 0.$ 
\end{remark}

\begin{theorem}[Asymptotic stability with only one non-trivial equilibrium]\label{Thm:Attractiveness_E1E2_1eq}\mbox{} Let $(i,j)\in\{1,2\}^2, j\neq i$, and assume that $\RR_{0,i}>1$. Then:
\begin{enumerate}
    \item  $E_i$ is  globally  asymptotically stable in $S=X_i\cap\partial X_j$.
    \item If moreover $\RR_{0,j}\leq 1$ then $E_i$ is globally asymptotically stable in $S=X_i$. 
\end{enumerate}
\end{theorem}

\begin{theorem}[Asymptotic stability with two non-trivial equilibria : non-critical case]\label{Thm:AttractivenessR0>1}
Let $i,j\in\{1,2\}$, $i\neq j$ and suppose $\min_{k\in\{1,2\}}\{\RR_{0,k}\}>1$.
\begin{enumerate}
    \item Suppose that $\gamma_{j\to i}>1$ and $\gamma_{i\to j}\leq 1$ then $E_i$ is globally asymptotically stable in $S=\cap_{k\in\{1,2\}} X_k$.
   \item Suppose that $\min_{k\in\{1,2\}}\{\gamma_{j\to i}\}>1$, then $E_3$ is globally  asymptotically stable in $S=\cap_{k\in\{1,2\}} X_k$.
\end{enumerate}
\end{theorem}

\begin{theorem}[Asymptotic stability with two non-trivial equilibria : critical case]\label{Thm:Attractiveness_continuum}
Suppose $\min_{k\in\{1,2\}}\{\RR_{0,k}\}>1$ and $\gamma_{1\to 2}=\gamma_{2\to 1}=1$. Then the continuum set of equilibria $E_\infty$ (defined in Proposition \ref{Prop:Positive_equilibria}) is globally asymptotically stable in  $S=\cap_{k\in\{1,2\}} X_k$. Moreover, for any $v=(x_0,y_{1,0},y_{2,0}) \in\X_1\cap X_2$, $\Phi_t(v)$ converges to $E_{1+\sigma}:= (1-\sigma)E_1 +\sigma E_2$ where $\sigma$ is the only solution in $(0,1)$ of $$(1-\sigma)y_1^*=\dfrac{y_{1,0}}{y_{2,0}}(\sigma y_2^*)^{\mu_1/\mu_2}.$$ 
\end{theorem}

One can note that the statements in Theorems \ref{Thm:Attractiveness_E0}, \ref{Thm:Attractiveness_E1E2_1eq}, \ref{Thm:AttractivenessR0>1} and \ref{Thm:Attractiveness_continuum} could actually be rewritten in a  stronger sense since in the proofs we show that the mentionned attractor $E$ ($E=E_i$ with $i=0,1,2,3$ or $\infty$ depending on the statement) actually attracts every compact subset $K$ of $S$, which is equivalent (\cite[Lemma 2.6]{SmithThieme2011}) to the fact that $d(\Phi_t(x),\{E\}) \to 0$ as $t\to +\infty$ uniformly for $x\in K$.

We can also note that if each competitor is able to survive in the absence of the other ($\RR_{0,i}>1$ for $i\in\{1,2\}$), then at least one competitor always survive, \textit{i.e.} the competition cannot drive both population to extinction.

Those results are summarized in Tables \ref{Table:R0<1_or_1competitor} and \ref{Table:R0>1_interior}.

\begin{table}[htbp]
\begin{center}
\begin{tabular}{|c|c|c|c|c|}
    \hline
    & $\partial  X_1\cap \partial  X_2$ & $\partial  X_1\cap  X_2$ & $ X_1\cap \partial  X_2$ & $ X_1\cap  X_2$ \\
   \hline
   $\max\{\RR_{0,1},\RR_{0,2}\}\leq 1$ & $E_0$ & $E_0$ & $E_0$ & $E_0$\\
   \hline
   $\RR_{0,1}>1\geq \RR_{0,2}$ & $E_0$ & $E_0$ & $E_1$ & $E_1$ \\
   \hline
   $\RR_{0,1}\leq 1<\RR_{0,2}$ & $E_0$  & $E_2$ & $E_0$  & $E_2$ \\
   \hline 
   $\min\{\RR_{0,1}, \RR_{0,2}\}> 1$ & $E_0$ & $E_2$ & $E_1$ & see table \ref{Table:R0>1_interior}\\
   \hline
\end{tabular}
\end{center}
\caption{Convergence of the solution of \eqref{Eq:model} depending on the basic reproduction numbers and the subset of initial conditions in $X_+$. Results come from Theorems \ref{Thm:Attractiveness_E0} and \ref{Thm:Attractiveness_E1E2_1eq}.}\label{Table:R0<1_or_1competitor}
\end{table}

\begin{table}[htbp]
\begin{center}
\begin{tabular}{|c|c|c|c|}
    \hline
   & $\gamma_{1\to 2}<1$ & $\gamma_{1\to 2}=1$ & $\gamma_{1\to 2}>1$\\
   \hline
   $\gamma_{2\to 1}<1$ & \cellcolor{gray!40} & \cellcolor{gray!40} & $E_2$ \\
   \hline
   $\gamma_{2\to 1}=1$ & \cellcolor{gray!40} & $E_{1+\sigma}\in E_\infty$   &   $E_2$\\
   \hline
   $\gamma_{2\to 1}>1$ &   $E_1$ & $E_1$ & $E_3$  \\
   \hline 
\end{tabular}
\end{center}
\caption{Convergence of the solution of \eqref{Eq:model} from any initial condition in $  X_1\cap X_2$ in the case $\min\{\RR_{0,1}, \RR_{0,2}\}> 1$.  Attractor results come from Theorems \ref{Thm:AttractivenessR0>1} and \ref{Thm:Attractiveness_continuum}, while Proposition \ref{Prop:Lyap_welldef} states that three grayed-out configurations are never met.}\label{Table:R0>1_interior}
\end{table}

\subsection{Examples and numerical simulations}
\begin{exemple}\label{Exp:analytic}
We assume that $\mu_1=\mu_2=\mu$, $c_1=c_2=c$,
$$\beta_1(a)=\mathbf{1}_{[\underline{\beta},\overline{\beta}]}(a)\widetilde{\beta}(a), \quad \beta_2(a)=\beta_1(a+\alpha), \quad \text{a.e. } a\geq 0$$
where $0\leq \underline{\beta}<\overline{\beta}$, $\alpha\in(-\infty,\underline{\beta})$ and $\widetilde{\beta}\in L^1_+(\R_+)$. Then we can compute with \eqref{Eq:DefR0i} the quantities
$$\RR_{0,1}=\dfrac{\Lambda c}{\mu}\int_{\underline{\beta}}^{\overline{\beta}}\widetilde{\beta}(a)e^{-\int_0^a \mu_x(s)\d s}\d a, \quad \RR_{0,2}(\alpha)=\dfrac{\Lambda c}{\mu}\int_{\underline{\beta}}^{\overline{\beta}}\widetilde{\beta}(a)e^{-\int_0^{a-\alpha} \mu_x(s)\d s}\d a.$$
It follows that the function $\alpha\longmapsto \RR_{0,2}(\alpha)$ is increasing, with $\RR_{0,2}(0)=\RR_{0,1}$. From \eqref{Eq:Def_y^*} we see that
$$\dfrac{\Lambda c}{\mu} \int_0^\infty \beta_1(a)e^{-\int_0^a \mu_x(s)\d s}e^{-y_1^* \int_0^a \beta_1(s)\d s} \d a=1=\dfrac{\Lambda c}{\mu} \int_0^\infty \beta_2(a)e^{-\int_0^a \mu_x(s)\d s}e^{-y_2^*(\alpha)\int_0^a \beta_2(s)\d s}\d a$$
which is equivalent to
$$\int_{\underline{\beta}}^{\overline{\beta}} \widetilde{\beta}(a)e^{-\int_0^a \mu_x(s)\d s}e^{-y_1^* \int_{\underline{\beta}}^a \widetilde{\beta}(s)\d s} \d a=\int_{\underline{\beta}}^{\overline{\beta}} \widetilde{\beta}(a)e^{-\int_0^{a-\alpha} \mu_x(s)\d s}e^{-y_2^*(\alpha)\int_{\underline{\beta}}^a \widetilde{\beta}(s)\d s}\d a.$$
Since the function $\alpha\longmapsto e^{-\int_0^{a-\alpha}\mu_x(s)\d s}$ is increasing for every $a\geq 0$ then the function $\alpha \longmapsto y_2^*(\alpha)$ is necessarily increasing, with $y_2^*(0)=y_1^*$. Finally \eqref{Eq:Def_gamma} allows us to compute 
$$\gamma_{2\to 1}(\alpha)=\dfrac{\Lambda c}{\mu}\int_{\underline{\beta}}^{\overline{\beta}}\widetilde{\beta}(a)e^{-\int_0^a \mu_x(s)\d s}e^{-y_2^*(\alpha)\int_{\underline{\beta}}^a \widetilde{\beta}(s)\d s}\d a$$
hence the function $\alpha \longmapsto \gamma_{2\to 1}(\alpha)$ is decreasing and
$$\gamma_{1\to 2}(\alpha)=\dfrac{\Lambda c}{\mu}\int_{\underline{\beta}}^{\overline{\beta}} \widetilde{\beta}(a)e^{-\int_0^{a-\alpha}\mu_x(s)\d s} e^{-y_1^*(\alpha)\int_{\underline{\beta}}^{\min\{a-\alpha,\overline{\beta}\}} \widetilde{\beta}(s)\d s}\d a$$
so the function $\alpha\longmapsto \gamma_{1\to 2}(\alpha)$ is increasing, with $\gamma_{1\to 2}(0)=\gamma_{2\to 1}(0)=1$.

This example proves that if $\beta_1$ and $\beta_2$ differ only by a translation, then only the first competitor to reach the resource can survive. Precisely, we see that if $\RR_{0,1}>1$ then when $\alpha=0$ we are in the case of Theorem \ref{Thm:Attractiveness_continuum}, that is the continuum set of equilibria $E_\infty$ exists and is GAS in $X_1\cap X_2$. When $\alpha>0$, then $\gamma_{1\to 2}(\alpha)>1>\gamma_{2\to 1}(\alpha)$ leading to the first case of Theorem \ref{Thm:AttractivenessR0>1} that is $E_2$ is GAS in $X_1\cap X_2$. On the contrary, if $\alpha<0$ then $\gamma_{1\to 2}(\alpha)<1<\gamma_{2\to 1}(\alpha)$ that is $E_1$ is GAS in $X_1\cap X_2$. Also, there exists $\underline{\alpha}<0$ such that for every $\alpha\leq \underline{\alpha}$ then $\RR_{0,2}(\alpha)\leq 1$ and $E_2$ ceases to exist.
\end{exemple}

\begin{exemple}\label{Exp:Bifurcation}
This example illustrates (numerically) a case  where the globally asymptotically stable (attractive in $X_1\cap X_2$) steady state bifurcate from $E_1$ to $E_3$ and then to $E_2$. As in the previous example, we assume that the two competitors  differ only in their respective function $\beta_i$, $i=1,2$, defined as
\begin{equation}\beta_1(a)=\frac{1}{5}\mathbf{1}_{[2,7]}(a), \quad \beta_2(a)=\frac{1}{10}\left(\mathbf{1}_{[0,\alpha]}(a) \frac{a}{\alpha} + \mathbf{1}_{(\alpha,20]}(a) \left(\frac{20-a}{20-\alpha}\right)\right) , \quad \text{a.e. } a\geq 0 \label{Eq:ExampleBeta}\end{equation}
where $\alpha \in [0,30]$. That is,  $\beta_1$ and $\beta_2$ both have unit $L^1-$norm, with $\beta_1$ being a gate function while $\beta_2$ is a triangle shaped function with mode $\alpha$. 

As illustrated on Figures \ref{Fig:Diagram} and \ref{Fig:Simulations}, as the mode $\alpha$ of $\beta_2$ varies, system \eqref{Eq:model} successively undergoes two transcritical bifurcations where $E_2$ and $E_1$ respectively merge and exchange their stability with $E_3$ at $\alpha\approx 9.6$ and $11.8$ respectively, while $E_3$ leaves the positive cone before and after those bifurcations.
\begin{figure}[h]
\begin{center}
\includegraphics[scale=0.8]{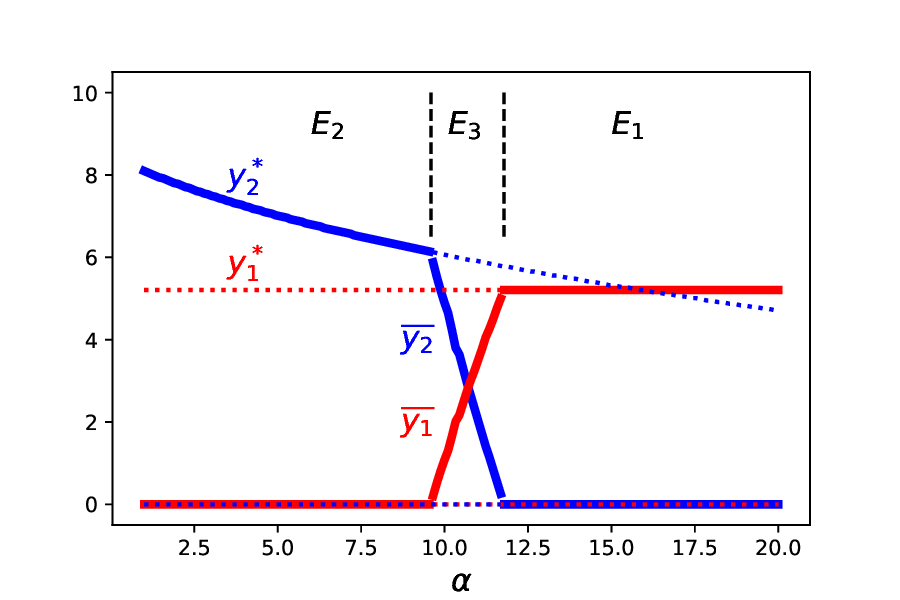}
\end{center}\caption{Bifurcation diagram for Example \ref{Exp:Bifurcation}. Only the  $y_1$ (red) and $y_2$ (blue) coordinates are depicted.  wide and thin lines correspond to stable and unstable equilibrium coordinates. The upper vertical dotted lines separate three intervals for $\alpha$ in which the global attractor in $X_1\cap X_2$ is $E_2,\ E_3$ or $E_1$. Parameters are (arbitrarily) set as follow : $c_1=c_2=\Lambda=1$, $\mu_1=\mu_2=0.1$ and $\mu_x\equiv 0.1$.}\label{Fig:Diagram}
\end{figure}

\begin{figure}[h]
\begin{center}
\includegraphics[scale=0.5]{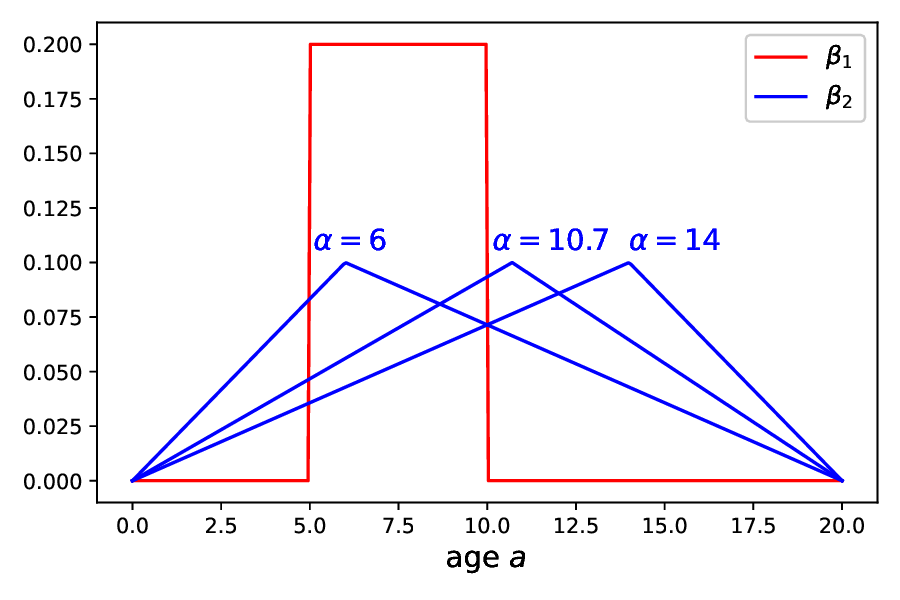}
\includegraphics[scale=0.5]{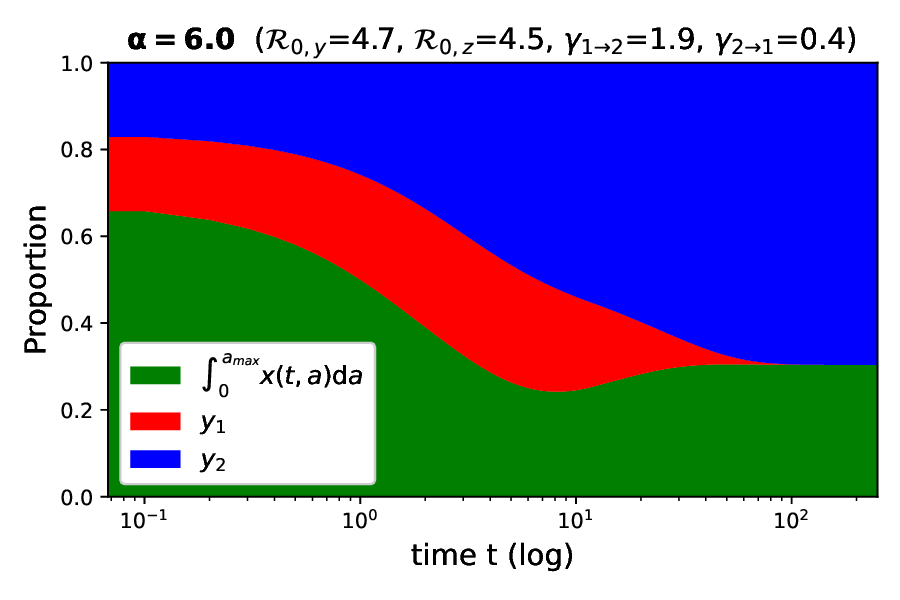}
\includegraphics[scale=0.5]{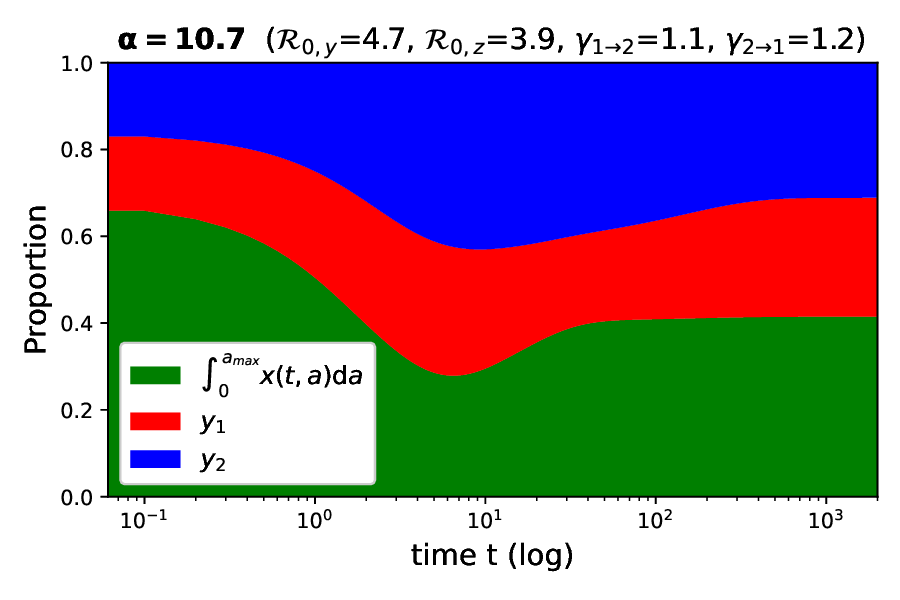}
\includegraphics[scale=0.5]{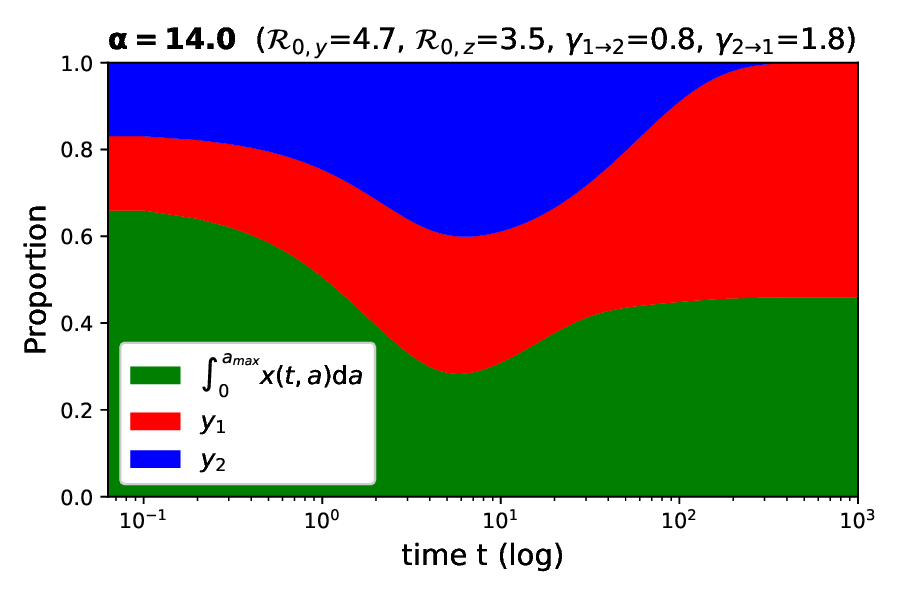}
\end{center}\caption{Top left) graphs of the functions $\beta_1$ and $\beta_2$ as defined in \eqref{Eq:ExampleBeta} for $\alpha=6,\ 10.7$ and $14$, corresponding to three areas of the bifurcation diagram in Figure \ref{Fig:Diagram}. Top right, bottom left and right) Numerical simulations for $\alpha=6,\ \alpha=10.7$ and $\alpha=14$ of the relative proportions of $y_1$, $y_2$ and $\int_0^{a_{\max}}x(t,a)\d a$, with $a_{\max}=20$, $x_0=e^{-0.05a}$ and $y_0=z_0=3$.}\label{Fig:Simulations}
\end{figure}

\end{exemple}

\section{Lyapunov functionals}\label{SectionProof_Lemma: CNS_eq_inf}

In this section, we introduce candidate Lyapunov functionals, that will be used in the next Section. Those functionals are not defined on the whole $X_+$ space, yet we prove that they are well-defined on all complete orbits (hence all invariant subsets) included in suitable sets  and that they are valid Lyapunov functionals.

\subsection{Definitions}

We first define the key function $g:\R_+^*\to \R_+$ by $g(z)=z-\ln(z)-1$, known as the Volterra Lyapunov function. Note that $g(z)=0$ if $z=1$ and $g(z)>0$ otherwise. Then we formally introduce the following functionals
$$L_0(x,y_1,y_2)=\int_0^\infty x_0^*(a)g\left(\frac{x(a)}{x_0^*(a)}\right)da+\frac{y_1}{c_1}+\frac{y_2}{c_2},$$
\begin{equation} L_i(x,y_1,y_2)=\int_0^\infty x_i^*(a)g\left(\frac{x(a)}{x_i^*(a)}\right)da+\frac{y_i^*}{c_i}g\left(\frac{y_i}{y_i^*}\right)+\frac{y_j}{c_j},\quad \forall i,j\in\{1,2\},\ j\neq i.\label{Eq:Def_Li}\end{equation}

In the case $\min(\RR_{0,1},\RR_{0,2})>1$ and  $\min(\gamma_{2\to 1},\gamma_{1\to 2})>1$, we know from Proposition \ref{Prop:Positive_equilibria} (first item) that there exists a positive equilibrium $E_3=(\overline{x},\overline{y_1},\overline{y_2})$. We define the associated functional
$$L_3(x,y_1,y_2)=\int_0^\infty \overline{x}(a)g\left(\frac{x(a))}{\overline{x}(a)}\right)\d a+\overset{2}{\underset{i=1}{\sum}}\frac{\overline{y_i}}{c_i} g\left(\frac{y_i}{\overline{y_i}}\right).$$
The domain of definition $D_{L_0}$, $D_{L_1}$, $D_{L_2}$ and $D_{L_3}$ of those functionals will be investigated in Lemma \ref{Lemma:Lyap_welldef_Orbit}. We remind the following definition
\begin{definition}[Lyapunov functional]
Let $S$ be a forward invariant subset of $X_+$. A function $L:S\to \R$ is called a Lyapunov functional on $S$ if:
\begin{itemize}
    \item $L$ is continuous on $S$;
    \item the function $\R_+\ni t\longmapsto L(\Phi_t(v))$ is non-increasing for every $v\in S$.
\end{itemize}
If moreover $E=\{v\in S\ : \ \R_+\ni t\longmapsto L(\Phi_t(v)) \text{ is constant }\}\neq \emptyset$, we say that $L$ is strict for $E$.

\end{definition}

\begin{definition}[Complete orbit]
A set $\gamma\subset   X_+$ is called a complete orbit for system \eqref{Eq:model} if there exists $\Psi:\mathbb{R}\to X_+$ such that $\gamma=\{\Psi(s),\ s\in \R\}$ and $\forall (t,s)\in\R_+\times \R$, $\Phi_t(\Psi(s))=\Psi(t+s)$. If $\eta\in\gamma$, we say that $\gamma$ is a complete orbit through $\eta$ and, without loss of generality, we can consider that $\Psi(0)=\eta$.
\end{definition}

Moreover, if $\{\Psi(s),\ s\in \R\}$ is a complete orbit through $\eta$, we define the set
$$\alpha(\eta)=\bigcap_{\tau \geq 0}\overline{\{\Psi(-t), \ t\geq \tau\}}.$$

\subsection{Well-posedness of Lyapunov functionals}

For $L_0$, $L_1$, $L_2$ and $L_3$ to be well defined and to be Lyapunov functionals, the components $(x,y_1,y_2)$ must satisfy positive and regularity properties. The following lemma states that those properties are typically satisfied on complete orbits, since it imposes that the characteristic solution originates from the positive boundary condition. It will be used in Section \ref{Sec:GlobalAnalysis} to prove the well-posedness of each functional $L_i$ for $i\in\llbracket 0,3\rrbracket$  on omega-limit sets (Propositions \ref{Prop:Lyap_welldef_1eq} and \ref{Prop:Lyap_welldef}) and to characterize the basins of attraction of all equilibria.

\begin{lemma}[Well-posedness of Lyapunov functionals on bounded complete orbits]\label{Lemma:Lyap_welldef_Orbit}\mbox{} 

\begin{enumerate} 
    \item Denote $\OO_+$ the set of all bounded and complete orbits included in $X_+$, then $\OO_+\subset D_{L_0}$, in particular $E_0\in  D_{L_0}$.  Moreover,  suppose that one of the following assumptions hold : 
    \begin{enumerate}
        \item $\max(\RR_{0,1},\RR_{0,2})\leq 1$, $S=X_+$,
        \item $\RR_{0,1}\leq 1$ and $S= \partial X_2$,
        \item $\RR_{0,2}\leq 1$ and $S= \partial X_1$,
        \item $S= \partial X_1\cap \partial X_2$,
    \end{enumerate}
    then $L_0$ is a Lyapunov functional on $\OO_+\cap S$, strict for $\{E_0\}$.
    \item Let $i,j\in\{1,2\}$ with $i\neq j$. Denote $\OO_i$  the set of all bounded and complete orbits included in $X_i$. If $\RR_{0,i}>1$, then $\OO_i\subset D_{L_i}$. In particular, $E_i\in D_{L_i}$. Moreover, suppose that one of the following assumptions holds : 
    \begin{enumerate}
        \item $\gamma_{i\to j} < 1$ and  $S=X_i$,
        \item $S= X_i\cap\partial X_j$,
        \item $\RR_{0,j}>1$, $\gamma_{i\to j} = 1$,  $\gamma_{j\to i}\neq 1$ and  $S=X_i$,
        \item $\RR_{0,j}>1$, $\gamma_{i\to j} = \gamma_{j\to i} = 1$ and  $S=X_i$,
    \end{enumerate}
    then $L_i$ is a Lyapunov functional on $\OO_i\cap S$ and $L_i$ is strict for $\{E_i\}$ if $(a)$, $(b)$ or $(c)$ is true and strict for $E_\infty\cup \{E_i\}$ if $(d)$ is true. 
    
    \item If $\min(\RR_{0,1},\RR_{0,2})> 1$, $\gamma_{2\to 1}>1$ and $\gamma_{1\to 2}>1$, then $E_3$ (as defined in Proposition \ref{Prop:Positive_equilibria}) is the only equilibrium in $X_1\cap X_2$ and $L_3$ is a well-defined Lyapunov functional on $\OO_1\cap\OO_2\subset D_{L_3}$, strict for $\{E_3\}$.

\end{enumerate}
\end{lemma}

Remark that, from \eqref{Eq:Def_gamma}, the assumption $\gamma_{i\to j} < 1$ in Lemma \ref{Lemma:Lyap_welldef_Orbit} is always satisfied if   $\RR_{0,j}\leq 1$. To prove Lemma \ref{Lemma:Lyap_welldef_Orbit} we need to introduce the following Lemma.

\begin{lemma}[Differentiability of the semi-flow along the complete orbits]\label{Lemma:Flow_C1_Orbit}\mbox{}
If there exists a complete orbit through $v\in X_+$, then $t\longmapsto\Phi_t(v)\in C^1(\R_+,X)$, \text{i.e.} is continuously differentiable in $X$. 
\end{lemma}
\begin{proof}[Proof of Lemma \ref{Lemma:Flow_C1_Orbit}]\mbox{}
Set $v=(x,y_1,y_2)\in X_+$ and $\tilde{v}=((0,x),y_1,y_2)\in \X_{0+}$. \cite[Theorem 5.6.6, p. 242]{MagalRuan2018} states that if $\tilde{v}\in D(A)$ and $A\tilde{v}+F(\tilde{v})\in \X_0$, then $t\longmapsto \Phi_t(v)$ is continuously differentiable in $L^1(\R_+,\R_+)\times\R_+^2$. This theorem also requires $F$ to be continuously differentiable, which is readily verified as each term in the expression of $F$ is a bilinear form.
 
Here, we prove that if there exists a complete orbit $\Psi=(\Psi^x,\Psi^1,\Psi^2)$ through $v$, then necessarily $\tilde{v}\in D(A)$ and $A\tilde{v}+F(\tilde{v})\in \X_0$. Indeed, let $\Psi$ be  a complete orbit through $v$, such that $\Psi(0)=v$, then for  $a\geq 0$ the method of characteristics implies that
$$x(a)=\Psi^x(0)(a)=\Lambda e^{-G(a)},\text{ with } G(a):= \int_0^a  \mu_x(s)+\sum_{i=1}^\n \beta_i(s)\Psi^{i}(s-a)  \mathrm{ds}.$$

Consequently, $x(a)\leq \Lambda e^{-a\mu_0}$, so $x\in L^1(\R_+,\R_+)$. Moreover, since $F$ is positive and absolutely continuous and since $t\longmapsto e^{-t}$ is Lipschitz continuous on $\R_+$, $x$ is also absolutely continuous, then differentiable almost everywhere. Finally, $x$ is clearly decreasing, then $$\int_0^{+\infty} |x'(a)|da =-\int_0^{+\infty}  x'(a) \d a =\underset{a\to +\infty}{\lim} x(0)-x(a)=\Lambda <+\infty.$$
    We just proved that $x\in W^{1,1}(\R_+)$, then $\tilde{v}\in D(A)$. Finally, $x(0)=\Lambda$ is readily verified, so $A\tilde{v}+F(\tilde{v})\in \X_0$.
\end{proof}

Now we can prove Lemma \ref{Lemma:Lyap_welldef_Orbit}.
\begin{proof}[Proof of Lemma \ref{Lemma:Lyap_welldef_Orbit}] \mbox{}
\begin{enumerate}
\item  First we show that $L_0$ is well defined on $\OO_+$.

Let $v\in X_+$ such that there exists a bounded and complete orbit $\Psi=(\Psi^x,\Psi^1,\Psi^2)$ which satisfies $\Psi(0)=v$.

For any $a\geq 0$, if $t>a$, then $\Psi^x(t)(a)=\Phi_t^x(\Psi(0))=\Phi^x_t(v)$. If $t\leq a$, then for any $b>a-t$, $\Psi^x(t)(a)=\Phi^x_{t+b}(\Psi(-b))(a)$. In both cases, formula \eqref{Eq:Duhamel} and the boundedness of $\Psi$ implies 
$$0<\Lambda e^{-\int_0^a \mu_x(s)\d s}e^{-a\|\beta_i\|_{L^\infty}\|\Psi^i\|_{L^\infty(\R)}} < \Psi^x(t)(a)=\Lambda e^{-\int_0^a \mu_x(s)+\sum_{i=1}^\n \beta_i(s)\Psi^{i}(t-a+s)  \d s}\leq \Lambda e^{-a\mu_0}$$  so that for any $t\in\R$, $a\longmapsto x_0^*(a)g(\Psi^x(t)(a)/x_0^*)$ is well defined on $\R_+$. From Assumption \ref{Assum:param}, simple comparisons  prove that the integral of the latter function, that appears in the definition of $L_0$, is well defined. 

Next we prove the Lyapunov decreasing property. Let $t\geq 0$ and set $(x(t,\cdot),y_1(t),y_2(t))=\Phi_t(v)$. From Lemma \ref{Lemma:Flow_C1_Orbit}, we can compute the following derivative
    \begin{flalign*}
        \frac{dL_0}{dt}(\Phi_t(v))=&\int_0^\infty \left(1-\frac{x_0^*(a)}{x(t,a)}\right)\partial_t x(t,a)\d a+\sum_{i=1}^\n \frac{y_i'(t)}{c_i}  \\
        =&-\int_0^\infty \left(1-\frac{x_0^*(a)}{x(t,a)}\right)\left(\partial_a x(t,a)+(\mu_x(a)+\sum_{i=1}^\n \frac{y_i(t)}{c_i}\beta_i(a))x(t,a)\right)\d a  \\
        &+\sum_{i=1}^\n y_i(t)\left(\int_0^\infty \beta_i(a)x(t,a)\d a-\frac{\mu_i}{c_i}\right).
    \end{flalign*}
    We use the fact that
    $$x_0^*(a)\frac{dg}{da}\left(\frac{x(t,a)}{x^*_0(a)}\right)=\left(1-\frac{x_0^*(a)}{x(t,a)}\right)\left(\partial_a x(t,a)-\frac{x(t,a){x_0^*}'(a)}{x_0^*(a)}\right)=\left(1-\frac{x_0^*(a)}{x(t,a)}\right)\left(\partial_a x(t,a)+\mu_x(a)x(t,a)\right)$$
    whence
    \begin{flalign*}
        \frac{dL_0}{dt}(\Phi_t(v))=&-\int_0^\infty x_0^*(a)\frac{\d g}{\d a}\left(\frac{x(t,a)}{x_0^*(a)}\right)\d a-\int_0^\infty\left(x(t,a)-x_0^*(a)\right) \sum_{i=1}^\n y_i(t)\beta_i(a)\d a\\
        &+\sum_{i=1}^\n y_i(t)\left(\int_0^\infty \beta_i(a)x(t,a)\d a-\frac{\mu_i}{c_i}\right)\\
        =& -\int_0^\infty \mu_x(a)x_0^*(a)g\left(\frac{x(t,a)}{x_0^*(a)}\right)da+\sum_{i=1}^\n y_i(t)\left(\int_0^\infty \beta_i(a)x_0^*(a)\d a-\frac{\mu_i}{c_i}\right).
    \end{flalign*}
    We finally end up with
    \begin{equation}
    \label{Eq:L0-deriv}
        \frac{dL_0}{dt}(\Phi_t(v))=-\int_0^\infty \mu_x(a)x_0^*(a)g\left(\frac{x(t,a)}{x_0^*(a)}\right)\d a+\sum_{i=1}^\n\frac{\mu_i}{c_i} y_i (t)(\RR_{0,i}-1)
    \end{equation}
    and we can actually show that
    $$\frac{dL_0}{dt}(\Phi_t(v))\leq -\int_0^\infty \mu_x(a)x_0^*(a)g\left(\frac{x(t,a)}{x_0^*(a)}\right)da\leq 0.$$
    Indeed, it is clear if $\underset{i\in\{1,2\}}{\max}\{\RR_{0,i}\}\leq 1$, which proves case (a). Finally (b), (c) and (d) are direct consequences of the positive invariance of $\partial  X_i$ for $i\in\{1,2\}$ (Proposition \ref{Prop:Attractors}) so that for any $i\in\{1,2\}$ and initial condition in $\partial  X_i$,  $y_i(t)=0$ for all $t\geq 0$. 
    
    Finally, suppose that  $t\longmapsto L_0(\Phi_t(v))$ is constant on $\R_+$. Then from equation \eqref{Eq:L0-deriv},   $\forall t\geq 0,\  x(t,.)=x_0^*$.  If moreover $R_{0,i}<1$ for some $i\in\{1,2\}$, then $y_i\equiv  0$. Otherwise, if  $R_{0,i}=1$, then injecting the relation $\forall t\in \R_+,\ x(t,.)= x_0^*$ in the first equation of \eqref{Eq:model} leads for every $t\in \R_+$ and a.e. $a\geq 0$ to $\beta_1(a)y_1(t)+\beta_2(a)y_2(t)=0$.
    
    Since $\beta_i\equiv 0$  would be a contradiction with $R_{0,i}=1$, we can conclude that $y_i\equiv 0$.  

    Consequently, $t\longmapsto L_0(\Phi_t(v))$ is constant on $\R_+$ if and only if $v=E_0$.

    \item Let $v\in X_1$ such that there is a bounded and complete orbit $\{\Psi(t),t\in \R\}$ through $v$ (such a $v$ exists, since the constant orbit $\{E_1\}$ is one). We prove just as above that if $\RR_{0,1}>1$, then for any $t\in\R$, $a\longmapsto x_1^*(a)g(\Psi^x(t)(a)/x_1^*)$ is well defined on $\R_+$ and is in $ L^1(\R_+,\R_+)$. The term $g(\Psi^1(t)/y_1^*)$ is also well defined since $X_1$ is invariant. So, for all $t\in \R,\ L_1(\Psi(t))$ is well defined. 
    
    Denote $(x(t,\cdot),y_1(t),y_2(t))=\Phi_t(v)$ and let $t\geq 0$, then from Lemma \ref{Lemma:Flow_C1_Orbit} we can compute the following derivative:
    \begin{flalign}
\frac{dL_1}{dt}(\Phi_t(v))=& -\int_0^\infty \left(1-\frac{x_1^*(a)}{x(t,a)}\right)(\partial_a x(t,a)+x(t,a)\left(\mu_x(a)+\sum_{i=1}^\n y_i(t)\beta_i(a)\right)\d a+\left(1-\frac{y_1^*}{y_1(t)}\right)\frac{y_1'(t)}{c_1}+\frac{y_2'(t)}{c_2} \nonumber \\
=& -\int_0^\infty x_1^*(a)\frac{dg}{da}\left(\frac{x(t,a)}{x_1^*(a)}\right) \d a-\int_0^\infty \left(1-\frac{x_1^*(a)}{x(t,a)}\right)(\beta_1(a)(y_1(t)-y_1^*)+y_2(t)\beta_2(a))x(t,a)\d a\nonumber \\
&+\frac{(y_1(t)-y_1^*)}{c_1}\left({c_1}\int_0^\infty \beta_1(a)x(t,a)\d a-\mu_1\right)+\frac{z(t)}{c_2}\left({c_2}\int_0^\infty \beta_2(a)x(t,a)\d a-\mu_2\right)\nonumber \\
=& -\int_0^\infty \left(\mu_x(a)+y_1^*\beta_1(a)\right)x_1^*(a)g\left(\frac{x(t,a)}{x_1^*(a)}\right)\d a+(y_1(t)-y_1^*)\left(\int_0^\infty \beta_1(a)x_1^*(a)da-\frac{\mu_1}{c_1}\right)\nonumber \\
&+y_2(t)\left(\int_0^\infty \beta_2(a)x_1^*(a)\d a-\frac{\mu_2}{c_2}\right)\nonumber \\
=&-\int_0^\infty \left(\mu_x(a)+y_1^*\beta_1(a)\right)x_1^*(a)g\left(\frac{x(t,a)}{x_1^*(a)}\right)da+y_2(t)\frac{\mu_2}{c_2}\left(\gamma_{1\to 2}-1\right), \label{Eq:L1prime}
\end{flalign}
where we used the equalities  
\begin{flalign*}
x_1^*(a)\frac{dg}{da}\left(\frac{x(t,a)}{x^*_1(a)}\right)&=\left(1-\frac{x_1^*(a)}{x(t,a)}\right)\left(\partial_a x(t,a)-\frac{x(t,a){x_1^*}'(a)}{x_1^*(a)}\right)\\
&=\left(1-\frac{x_1^*(a)}{x(t,a)}\right)\left(\partial_a x(t,a)+(\mu_x(a)+\beta_1(a)y_1^*)x(t,a)\right)
\end{flalign*}
and
$${c_1}\int_0^\infty \beta_1(a)x^*_1(a)da=\mu_1, \qquad {c_2}\int_0^\infty \beta_2(a)x_1^*(a)da=\mu_2\gamma_{1\to 2}.$$

We first consider the case $\gamma_{1\to 2}<1$.  Remember that $g(x/x_1^*)=0$ if and only if $x=x_1^*$, so
$$\frac{dL_1}{dt}(\Phi_t(v))\leq 0,$$
with equality if and only if $x(t,\cdot)=x_1^*$ and $y_2(t)=0$, in particular $\frac{dL_1}{dt}(\Phi_t(E_1))=0$. 

Note that if $x(t,\cdot)=x_1^*$,  $y_2(t)=0$ and $y_1(t)\neq y_1^*$, then $\partial_tx(t,a)\neq 0$ so there is no   $v\in X_+\backslash\{E_1\}$ such that for all $t\geq 0$, $\frac{dL_1}{dt}(\Phi_t(v))=0$. We then deduce (a) and (b).

Now we consider the case $\RR_{0,2}>1, \gamma_{1\to 2}=1$ and $\gamma_{2\to 1}\neq 1$. In this case, $L_1$ is constant on the positive orbit $\{\Phi_t(v), t\geq 0\}$ if and only if $x(t,\cdot)\equiv x_1^*$. Then the two first equations in \eqref{Eq:model}, along with \eqref{Eq:Def_y^*}, lead, for all $t\geq 0$ and a.e. $a\geq 0$, to $$y_2(t)\beta_2(a)=(y_1^*-y_1(t))\beta_1(a) \quad \text{ and } \quad y_1'(t)=0.$$

In particular, $y_1$ is constant. Note that $\beta_2\not\equiv 0$, otherwise it would be a contradiction with  $\gamma_{1\to 2} =1$, then $y_2$ is also constant. Consequently, either  for all $t\geq 0$,  $y_1(t)=y_1^*$ and $y_2(t)=0$, then $\{\Phi_t(v), t\geq 0\}=\{E_1\}$ is the only orbit on which $L_1$ is constant and the proof is complete, or for all $t\geq 0$, $y_1(t)\neq y_1^*$. The latter case implies that for all $t\geq 0$ and $a\geq 0$,
$$\beta_1(a)=\delta \beta_2(a) \text{ where }  \delta=\frac{y_2(t)}{y_1^*-y_1(t)}\text{ is constant.}$$

Moreover $\delta\neq 0$, otherwise we would have $\beta_1\equiv 0$, which is not compatible with $\RR_{0,1}>1$.

Since $\gamma_{1\to 2}=1$ and $\beta_1=\delta\beta_2$, the definition of $\gamma_{1\to 2}$ in \eqref{Eq:Def_gamma} and equation \eqref{Eq:Def_y^*} lead to $\delta=\frac{\mu_1c_2}{\mu_2c_1}$.

Injecting the relation $\beta_1=\mu_1\beta_2/\mu_2$ into the definition of $\gamma_{2\to 1}$ ultimately gives $\gamma_{2\to 1}=1$, which is absurd. Then $\Phi_t(v)\equiv E_1$. 

Finally, consider the case (d) with $\RR_{0,j}>1$ and $\gamma_{i\to j} = \gamma_{j\to i} = 1$. Just as above,  $L_1$ is constant on the positive orbit $\{\Phi_t(v), t\geq 0\}$ if and only if $x(t,\cdot)\equiv x_1^*$. Remember that $x_1^*=x_2^*$ when $\gamma_{1\to 2}=\gamma_{2\to 1}=1$ (by Proposition \ref{Prop:Positive_equilibria}). Then injecting \eqref{Eq:Def_y^*} in \eqref{Eq:model} implies that for all $t\in \R$, $y_1'(t) = y_2'(t)= 0$. Consequently $\gamma$ is a constant orbit in $X_1$, \textit{i.e.} a positive equilibrium, then belongs to $E_\infty\cup \{E_1\}$.

This achieves the proof, and the same applies for the $L_2$ case.

\item The existence of an equilibrium $E_3=(\overline{x},\overline{y_1},\overline{y_2})\in \cap_{k\in\{1,2\}} X_k$ has been established in Section \ref{Sec:equilibria}. The well-posedness of $L_3$ on $\OO_1\cap \OO_2$ is obtained exactly as above for the one of $L_1$ and $L_2$. 
Let $v\in X_1\cap X_2$ such that there exists a bounded and complete orbite through $v$ (such a $v$ exists as the constant orbit $\{E_3\}$ is one). Denote $(x(t,\cdot),y_1(t),y_2(t))=\Phi_t(v)$ and let $t\geq 0$, then from Lemma \ref{Lemma:Flow_C1_Orbit}, we can compute the following derivative.

If $v=E_3$, then  $t\longmapsto L_3(\Phi_t(V))$ is readily constant, it remains to prove that the reciprocal is true. 
     
     Indeed, differentiating as in the proof of the second item leads to $$\frac{\d L_3}{\d t}(\Phi_t(v))=-\int_0^\infty \left(\mu_x(a)+\underset{k\in\{1,2\}}{\sum}\beta_k(a)\overline{y_k}\right)\overline{x}(a)g\left(\frac{x(t,a)}{\overline{x}(a)}\right)\d a\leq 0.$$

    Assume $t\longmapsto L_3(\Phi_t(v))$ is constant on $\R_+$. Then the previous derivaritives implies that for any $t\geq 0$, $\Phi_t^x(v)=\overline{x}$.
    
    Moreover, the definition of $\overline{x}$ (equation \eqref{Eq:CNS_positive_eq}) implies that, for any $k\in \{1,2\}$, $$c_k \int_0^\infty \beta_k(a) \overline{x}(a) \d a = \mu_k.$$
    
    Consequently, system \eqref{Eq:model}, implies that $y_1$ and $y_2$ are constant and that, for any $a\geq 0$,  
    $$\underset{k\in\{1,2\}}{\sum} y_k(t)\beta_k(a)  =\underset{k\in\{1,2\}}{\sum}\overline{y_k}\beta_k(a).$$

    Suppose that $y_1\not \equiv \overline{y_1}$. Then for any $a\geq 0$, $$\beta_1(a)=\delta \beta_2(a) \text{ with }  \delta=-\frac{y_2-\overline{y_2}}{y_1-\overline{y_1}}.$$
    
    Note that $\gamma_{2\to 1}>1$ implies that $\delta \neq 0$. Indeed, otherwise we would have $\beta_1\equiv 0$ and then $\gamma_{2\to 1}=0$. 
    
    Injecting this relation in the definition \eqref{Eq:Def_gamma} of $\gamma_{2\to 1}$  leads to $$\gamma_{2\to 1}=\dfrac{c_1\delta}{\mu_1}\int_0^\infty \beta_2(a)x_2^*(a)\d a = \dfrac{c_1\mu_2}{c_2\mu_1}\delta$$ where the last equality results from the definition \eqref{Eq:Def_y^*} of $y_2^*$.
    Similarly $\gamma_{1\to 2}=\dfrac{c_2\mu_1}{c_1\mu_2}\dfrac{1}{\delta}=1/\gamma_{2\to 1} $. Consequently $\gamma_{1\to 2}\gamma_{2\to 1}=1$, which is absurd and then $t\longmapsto L_3(\Phi_t(v))$ is constant if and only if $v=E_3$. 
    
    Consequently, the only positive orbit on which  $t\longmapsto L_3(\Phi_t)$ is constant is $\{E_3\}$, that is, $L_3$ is a Lyapunov functional on $\omega(v)$, strict for $E_3$.
    
    Since a positive equilibrium is, in particular, a positive orbit, the former proof also implies that $E_3$ is the only equilibrium in $X_1\cap X_2$.
\end{enumerate}
\end{proof}

\section {Global analysis and proof of the main results}\label{Sec:GlobalAnalysis}

In this section we handle the attractivity of the equilibria and compute their basin of attraction. To prove the stability, we will actually prove that $\{E_i\}$ not only attracts the points, but also the compact sets in their basin. We proved in Lemma \ref{Lemma:omega-limit} that for any $v\in X_+$, $\Phi_t(v)$ converges to $\omega(v)$. In this section, we prove that, under suitable conditions, each compact set $K$ is attracted by an invariant and compact set $\omega(K)$, on which the Lyapunov functionals $L_i$ (for $i\in\llbracket 0,3\rrbracket$) are well-defined, and which contains their critical point $E_i$. In particular, this requires to prove the uniform persistence of the flow to guarantee that $y_i$ does not vanish on $\omega(K)$, so that $\omega(K)$ belongs to the suitable subset identified in Lemma \ref{Lemma:Lyap_welldef_Orbit}.

It may be tempting to conclude that $\omega(K)=\{E_i\}$ once one has shown that $\omega(w)=\{E_i\}$ for every $w\in \omega(K)$, however this is not true in general. In other words, the convergence to $\omega(K)$ and then to $\omega(w)$ is not transitive as we only get $\omega(w)\subset \omega(K)$, even if $\omega(K)$ is compact and invariant (see \cite[Example p.224]{Teschl2012} for a simple counter-example). Instead, if we can define a Lyapunov functional on $\omega(K)$, strict for $\{E_i\}$, then $\omega(K)$ reduces to $\{E_i\}$. This is a special case of Theorem 2.52 from \cite{SmithThieme2011}, which generalises the LaSalle invariance principle and implies that a strict Lyapunov functional cannot be defined on an invariant compact set larger than the singleton reduced to its critical point. 

\subsection{Proof of Lemma \ref{Lemma: CNS_eq_inf}}

\begin{description}
\item[Proof of $(3)\Rightarrow (1)$.] This is a straightforward consequence of the definition of $y_1^*$ and $y_2^*$.

\item[Proof of $(1)\Rightarrow(2)$.] Let $v=(x_0,y_{1,0},y_{2,0})\in  X_1\cap  X_2$ where for all $a\geq 0$,  $x_0(a)= \Lambda e^{-a}$. Then straightforward calculations prove that $L_1$ and $L_2$ are well-defined in $v$ and on the positive orbit starting from $\{v\}$. Since $x_0$ is, in particular, in $ W^{1,1}(\R_+)$ then \cite[Theorem 5.6.6, p. 242]{MagalRuan2018} applies and we can differentiate $t\longmapsto L_1(\Phi_t(v))$ as is the proof of Lemma  \ref{Lemma:Lyap_welldef_Orbit}. It follows from \eqref{Eq:L1prime} that $t\longmapsto L_1(\Phi_t(v))$ is decreasing, then converges, so its derivative converges to $0$, that is $\Phi_t^x(v)$ converges to $x_1^*$ in $L^1(\R_+,\R_+)$. The same arguments applied to $L_2$ show that $\Phi_t^x(v)$ converges to $x_2^*$ in $L^1(\R_+,\R_+)$. Consequently  $x_1^*=x_2^*$  in $L^1(\R_+)$, that is, for a.e. $a\geq 0$:
$$\Lambda e^{-\int_0^a \mu_x(s)ds}e^{-y_1^*\int_0^a \beta_1(s)\d s}=\Lambda e^{-\int_0^a\mu_x(s)ds}e^{-y_2^*\int_0^a \beta_2(s)\d s}$$ 
then
$$ y_1^*\int_0^a \beta_1(s)\d s = y_2^*\int_0^a \beta_2(s)\d s$$ 
and finally $y_1^* \beta_1(a)=y_2^* \beta_2(a)$ a.e. $a\geq 0$.

\item[Proof of $(2)\Rightarrow(3)$.]  Equation \eqref{Eq:Def_y^*} with $i=1$ states that
$$c_1\int_0^\infty\beta_1(a)x_1^*(a)\d a=c_1\Lambda\int_0^\infty\beta_1(a) e^{-\int_0^a \mu_x(s)\d s}e^{-y_1^* \int_0^a \beta_1(s)\d s}\d a = \mu_{1}.$$
    
With the second item of Lemma \ref{Lemma: CNS_eq_inf}, it rewrites as   $$\frac{c_1\Lambda y_2^*}{y_1^*}\int_0^\infty\beta_2(a) e^{-\int_0^a \mu_x(s)ds}e^{-y_2^*\int_0^a \beta_2(s)\d s}\d a = \frac{c_1\Lambda y_2^*}{y_1^*}\int_0^\infty\beta_2(a) x_2^*(a)\d a = \mu_{1}$$
    
Using equation \eqref{Eq:Def_y^*}, now with $i=2$, we obtain
$$c_1 y_2^*\mu_2=c_2\mu_1 y_1^*.$$
Since $y_2^*\beta_2(a)=y_1^*\beta_1(a)$ almost everywhere $a\geq 0$, multiplying each side by $\beta_2$ leads to the third item of Lemma \ref{Lemma: CNS_eq_inf}.\hfill $\square$
\end{description}

\subsection{Proof of Theorem \ref{Thm:Attractiveness_E0} : Attractivity of $E_0$}

We introduce the local notations $L_*=L_0$ and $E_*=E_0$, so that the same proof can be reused later for different settings. 

The proof proceed in two steps. First we prove that every compact set $K\in S$ is attracted by a compact subset of $\omega(K)\subset S$. The second step basically follows the proof of \cite[Theorem 2.52]{SmithThieme2011} to prove that $\omega(K)$ reduces to $E_*$

Let $K\subset S$ be a compact set. From Proposition \ref{Prop:solutions} $\Phi$ is asymptotically smooth on $X_+$. From equation \eqref{Eq:MassMajoration}, $t\longmapsto \|\Phi_t(K)\|_{X}$ is bounded on $\R_+$ (in the vocabulary of \cite{SmithThieme2011},  $\Phi$ is eventually bounded on $K$). Then \cite[Proposition 2.27]{SmithThieme2011} states that $\omega(K):=\cap_{t\geq 0}\overline{\{\Phi_s(K),s\geq t\}}$ is a compact invariant set that attracts $K$. Moreover, since $S$ is invariant and closed, then $\omega(K)\subset S$.

Next, we prove that $\omega(K)=\{E_*\}$. Since $\omega(K)$ is invariant and compact, there exists a bounded and complete orbit through any element of it. Then we know from Lemma \ref{Lemma:Lyap_welldef_Orbit} that $L_*$ is a well defined Lyapunov functional on $\omega(K)$. Let $v\in \omega(K)$ and $\gamma(v)=\{\Psi(t),t\in \R\}\subset \omega(K)$ be a complete and bounded orbit through $v$. The compactness of $\omega(K)$ implies that $\omega(v)\subset \omega(K)$ and $\alpha(v)\subset \omega(K)$. It follows from \cite[Proposition 2.51]{SmithThieme2011} that $L_*$ is constant on the invariant sets $\alpha(v)$ and $\omega(v)$, in particular $t\longmapsto L_*\circ \Phi_t(u)$ is constant for any $u\in \alpha(v) \cup \omega(v)$. Then Lemma \ref{Lemma:Lyap_welldef_Orbit}  implies that necessarily $u\in\{E_*\}$, so that $\omega(v)=\alpha(v)=\{E_*\}$, in particular $\{E_*\}\subset \omega(K)$. Consequently $\lim_{t\to-\infty}L_*(\Psi(t))=\lim_{t\to+\infty}L_*(\Psi(t))=L_*(E_*)$. This implies that $v=E_*$, indeed $t\longmapsto L_*\circ \Phi_t(u)$ is decreasing for any $u\in \omega(K)\backslash\{ E_*\}$ (Lemma \ref{Lemma:Lyap_welldef_Orbit}). Then $\omega(K)=\{E_*\}$. 
We proved that $\{E_*\}=\omega(K)$ attracts $K$, then $\{E_*\}$ attracts every compact set of $S$.

We just proved that $\{E_*\}$ is a compact attractor of compact sets  (in particular, an attractor of points) in $S$. Since $\Phi$ is also state-continuous, uniformly in finite time (Proposition \ref{Prop:solutions}),  \cite[Theorem 2.39]{SmithThieme2011} states that $\{E_*\}$ is stable. \hfill $\square$

\subsection{Proof of Theorem \ref{Thm:Attractiveness_E1E2_1eq} : case with only one non-trivial equilibrium}
We characterise the attractivity of $E_1$ and $E_2$ in the case where only one of them exists (\textit{i.e.} $\RR_{0,1}>1\geq \RR_{0,2}$ or $\RR_{0,2}>1\geq \RR_{0,1}$), or when we restrict the dynamics to a one-predator model (\textit{i.e.} to the boundary $\partial X_1$ or $\partial X_2$).

\begin{definition}[persistence \cite{SmithThieme2011}]\label{Def:persistence}
Let $\rho: X_+\to\R_+$. The semiflow $\Phi$ is uniformly weakly $\rho-$persistent if there exists  $\ep >0$ such that  $$u\in  X_+, \rho (u)>0 \Longrightarrow   \underset{t\to +\infty}{\limsup} \ \rho (\Phi_t(u))\geq \ep .$$ 
If we can replace $\limsup$ with $\liminf$ in the former, then $\Phi$ is said to be uniformly $\rho-$persistent.
\end{definition}

We want to show that under suitable conditions, for any $v\in X_i$ with $i\in\{1,2\}$ then $\omega(v)\subset X_i$, so that Lemma \ref{Lemma:Lyap_welldef_Orbit} implies that the appropriate Lyapunov functional is well defined on $\omega(v)$. To do so, we make use of \cite[Theorem 8.20, p.189]{SmithThieme2011} (see also \cite{HaleWaltman89}) to conclude on the uniform persistence of the flow based only on its behaviour on the boundary set $\rho^{-1}(\{0\})$ (here $\partial X_1$ and $\partial X_2$) and near the attractor included in this set.

\begin{proposition}[Well-posedness of Lyapunov functional on omega-limit sets : case with only one non trivial equilibrium]\label{Prop:Lyap_welldef_1eq}\mbox{}
Let $(i,j)\in\{1,2\}^2, j\neq i$ and suppose that $\RR_{0,i}>1$. Define  $\rho_i:X_+\to \R_+,\ (x,y_1,y_2)\longmapsto y_i$. Then the restriction of the semiflow $\Phi$ to the invariant set $\partial X_j$ is uniformly $\rho_i$-persistent. If moreover  $\RR_{0,j}\leq 1$ then $\Phi$ is uniformly $\rho_i$-persistent.
\end{proposition}

\begin{proof}
We only provide a proof for the uniform persistence on $X_+$ when $\RR_{0,i}>1$ and $\RR_{0,j}\leq 1$. The other case is similar. 

From Theorem \ref{Thm:Attractiveness_E0}, $\{E_0\}$ is globally attractive in  $\partial  X_i$ (since $\RR_{0,j}\leq 1$) and consequently $\underset{u\in \partial  X_i}{\bigcup}\omega(u)=\{E_0\}$. Below, we prove the four required assumptions to apply \cite[Theorem 8.20]{SmithThieme2011}, which states that $\Phi$ is uniformly weakly $\rho_i$-persistent.
    
    \begin{description}
    \item[1) $\{E_0\}$ is trivially compact and invariant.] 
    \item[2) $E_0$ is isolated in $\partial  X_i$,] \textit{i.e.} the maximal compact invariant set of a neighborhood of itself in $\partial  X_i$.  Actually we already proved in the proof of Theorem \ref{Thm:Attractiveness_E0} the stronger property that $\{E_0\}$ is the only compact invariant set in $\partial  X_i$. 

    \item[3) $\{E_0\}$ is acyclic in $\partial  X_i$,] \textit{i.e.}  apart from the constant orbit $\{E_0\}$, there is no complete orbit $\gamma \subset \partial  X_i$ connecting $\{E_0\}$ to itself (a homoclinic loop). Indeed, consider $u\in \partial  X_i$, $u\neq E_0$  and $\gamma(u)$ a homoclinic loop through $u$ in $\partial  X_i$. Then the proof from the previous point applies and $\gamma(u)=\{E_0\}$.
    \item[4) $\{E_0\}$ is uniformly weakly repelling in $ X_i$,] \textit{i.e.} $\exists$ $\ep >0$ such that $\underset{t\to +\infty}{\lim \sup}\|E_0-\Phi_t(u)\|_{ X}\geq \ep $ for any $u\in  X_i$. Indeed, let $u\in  X_i$ and denote $\Phi_t(u)=(x(t),y_1(t),y_2(t))$, then for any $t\geq 0$:
    \begin{flalign*}
     y_i'(t)=&y_i(t)\left(c_i\int_0^\infty \beta_i(a)(x(t,a)-x_0^*(a))\d a+c_i\int_0^\infty \beta_i(a)x_0^*(a)\d a-\mu_{i}\right)\\
     =& y_i(t)\left(c_i\int_0^\infty \beta_i(a)(x(t,a)-x_0^*(a))\d a +\mu_i(\RR_{0,i}-1)\right).
     \end{flalign*}
     Set $\ep =\frac{\mu_i(\RR_{0,i}-1)}{2 c_i ||\beta_i||_{L^\infty}}$ and assume that $\underset{t\to +\infty}{\lim \sup}\|E_0-\Phi_t(u)\|_{ X}< \ep $. Then $\exists$ $T>0$, $\forall t\geq T$,  $|x(t,\cdot)-x_0^*|_{L^1}<\ep  $ implying that $\forall t\geq T$, $y_i'(t)>y_i(t)\left(\mu_i(\RR_{0,i}-1)-c_i||\beta_i||_{L^\infty}\ep \right)$, with $\mu_i(\RR_{0,1}-1)-c_i ||\beta_i||_{L^\infty}\ep >0$. We would get $\lim_{t\to +\infty}y_i(t)=+\infty$, which contradicts $\limsup_{t\to+\infty}\|E_0-\Phi_t(u)\|_{ X}< \ep $.   
    \end{description}
    
    Then $\Phi$ is uniformly weakly  $\rho_i$-persistent  and the existence of a global attractor (by corollary \ref{Cor:attractor}) implies that it is actually uniformly $\rho_i$-persistent (\cite[Theorem 5.2 and hypothesis H0 p.125]{SmithThieme2011}).
\end{proof}

Proposition \ref{Prop:Lyap_welldef_1eq} allows to conclude on the attractivity of positive equilibrium in the special case where there exists only one non-trivial equilibrium, or when we restrict the dynamics to a  one-predator model, as stated in Theorem \ref{Thm:Attractiveness_E1E2_1eq} proved below.

\begin{proof}[Proof of Theorem \ref{Thm:Attractiveness_E1E2_1eq}]
    The proof proceeds exactly as for Theorem \ref{Thm:Attractiveness_E0}, with $L_*=L_i$ and $E_*=E_i$, except that here $S$ is not closed. In this case, the inclusion $\omega(K)\subset S$  is obtained because $S$ is invariant and $\Phi$ is uniformly $\rho_i$-persistent (Proposition \ref{Prop:Lyap_welldef_1eq}), so that for any compact $K\subset S$, $\omega(K)\subset \{(x,y_1,y_2)\in S\ : \ y_i\geq \ep\}\subset S$, where $\ep$ is the one from the definition \ref{Def:persistence} of the uniform persistence.   
    
    Note that for the second item, one requires the additional hypothesis $\gamma_{i\to j}<1$ to guarantee  that $L_*$ is a Lyapunov functional, strict for $\{E_i\}$, on the compact invariant set $\omega(K)\subset S=X_i$ (Lemma \ref{Lemma:Lyap_welldef_Orbit}.2.a). This is readily satisfied since $\RR_{0,i}>1$ implies that $\|\beta_i\|_{L^1}> 0$, then $\gamma_{i\to j}<\RR_{0,j}\leq 1$ (Remark \ref{Remark: Comparison_gamma_R}). 
\end{proof}

\subsection{Proofs of Theorems \ref{Thm:AttractivenessR0>1} and \ref{Thm:Attractiveness_continuum} : case with two non-trivial equilibria}
In this case, at least two non-trivial equilibria exist and the following proposition takes the role of Proposition \ref{Prop:Lyap_welldef_1eq} 
\begin{proposition}[Well-posedness of Lyapunov functional on omega-limit sets : case with multiple non trivial equilibria]\label{Prop:Lyap_welldef}\mbox{}

Let $(i,j)\in\{1,2\}^2$, $i\neq j$  and suppose $\underset{i\in\{1,2\}}{\min}\{\RR_{0,i}\}>1$. Define $\rho_i:X_+\to \R_+,\ (x,y_1,y_2)\longmapsto y_i$,  $\rho_{1,2}:X_+\to \R_+,\ (x,y_1,y_2)\longmapsto  \min(y_1,y_2)$ and $\rho_{1+2}:X_+\to \R_+,\ (x,y_1,y_2)\longmapsto  y_1+y_2$. Then :
\begin{enumerate}

    \item $\gamma_{j\to i}\geq 1$ or $\gamma_{i\to j}\geq 1$.
    \item If $\gamma_{j\to i}=1$, then $\gamma_{i\to j}\geq 1$. 
    \item  If  $\gamma_{j\to i}>1$ then the semiflow $\Phi$ is uniformly $\rho_i$-persistent. 
    
    \item If  $\gamma_{j\to i}>1$ and $\gamma_{i\to j}>1$ then the semiflow $\Phi$ is uniformly $\rho_{1,2}$-persistent.   
    \item If  $\gamma_{1\to 2}=\gamma_{2\to 1}=1$, then the semiflow $\Phi$ is uniformly $\rho_{1+2}$-persistent.

\end{enumerate}
\end{proposition}

\begin{proof}\mbox{}
\begin{enumerate}
    \item Suppose that $\gamma_{j\to i}<1$ and $\gamma_{i\to j}<1$, then, exactly as for the case $\gamma_{j\to i}>1$ and $\gamma_{i\to j}>1$ (see Proposition \ref{Prop:Positive_equilibria}.1), there would exist at least one equilibrium $E_3\in  \cap_{k\in\{1,2\}}X_k$. Moreover, since $\{E_3\}$ is a complete orbit included in $\cap_{k\in\{1,2\}}X_k\subset X_i$, Lemma \ref{Lemma:Lyap_welldef_Orbit} states that $L_i$ is well defined  on $\{E_3\}$. Since $E_3$ is  an equilibrium, the function $t\longmapsto L_i(\Phi_t(E_3))$ is constant. On the other hand, since $\gamma_{i\to j}< 1$, $L_i$ is a Lyapunov functional on the complete (and constant) orbit through $E_3$, strict for $E_i$, and then $t\longmapsto L_i(\Phi_t(E_3))$ is decreasing. This is absurd, then it is not possible that $\gamma_{j\to i}<1$ and $\gamma_{i\to j}<1$.
    
    \item  Suppose that $\gamma_{j\to i}=1$ and $\gamma_{i\to j}< 1$. Set $v=(x_0,y_{1,0},y_{2,0})\in \cap_{k\in\{1,2\}} X_k$ where for all $a\geq 0$,  $x_0(a)= \Lambda e^{-a}$. Then, exactly as in the proof of Lemma \ref{Lemma: CNS_eq_inf}, we conclude that  $\Phi_t^x(v)$ converges to $x_i^*=x_j^*$ on $L^1(\R_+,\R_+)$, that $t\longmapsto L_j(\Phi_t(v))$ and $t\longmapsto L_i(\Phi_t(v))$ are decreasing and that their derivatives converge to zero. In particular, the latter implies that $\Phi_t^{j}(v)$ converges to zero (equation \eqref{Eq:L1prime}). In turn, this would imply that $\lim_{t\to +\infty}L_j(\Phi_t(v)) = +\infty$, by definition of $L_j$ in \eqref{Eq:Def_Li}, which is absurd since $t\longmapsto L_j(\Phi_t(v))$ is decreasing. 
    
    \item Let $v\in X_i$. We use again \cite[Theorem 8.20]{SmithThieme2011} to prove that $\omega(v)\subset  X_i$.

    From Theorem \ref{Thm:Attractiveness_E1E2_1eq}, we have $\underset{u\in \partial  X_i}{\bigcup}\omega(u)=\{E_0,E_j\}$ and we prove the following statements \begin{description}
    \item[1) $\{E_0\}$ and $\{E_j\}$ are trivially two disjoint, compact and invariant sets.]
    \item[2) $\{E_0,E_j\}$ is acyclic in $\partial X_i$,] \textit{i.e.} : \begin{itemize}
        \item Apart from the constant orbit $\{E_0\}$, there is no complete orbit $\gamma \subset \partial  X_i$ connecting $\{E_0\}$ to itself (a homoclinic loop).  Indeed, otherwise we would have $\gamma\subset \cap_{k\in\{1,2\}}\partial X_k$ since $\{E_j\}$ is globally attractive in $\partial X_i\cap X_j$ (Theorem \ref{Thm:Attractiveness_E1E2_1eq}). Moreover, $\cap_{k\in\{1,2\}} \partial X_k$ is invariant and equation \eqref{Eq:Duhamel} induces that the singleton orbit $\{E_0\}$ is the only complete orbit in $\cap_{k\in\{1,2\}} \partial X_k$.  
        
        \item Let $\gamma=\{\Psi(t),t\in\R\}\subset \partial X_i$ be a homoclinic loop connecting $E_j$ to itself. We prove that necessarily $\gamma=\{E_j\}$. We know that $E_j\in X_j$ and $\partial X_i\cap X_j$ is invariant. Since $\{E_0\}$ is globally attractive in $\partial X_i\cap \partial X_j$ (Theorem \ref{Thm:Attractiveness_E0}) then $\gamma\subset \partial X_i\cap X_j$. Moreover, by definition of $\gamma$,  $\lim_{t\to-\infty}L_j(\Psi(t))=\lim_{t\to+\infty}L_j(\Psi(t))=L_j(E_j)$. Since $L_j$ is a Lyapunov functional on $\gamma$ strict for $\{E_j\}$ (Lemma \ref{Lemma:Lyap_welldef_Orbit}.2.b), then $\gamma=\{E_j\}$, otherwise $L_j$ would have been decreasing along $\gamma$.
        \item There is no cycle formed by two complete orbits in  $\partial X_i$ connecting $\{E_0\}$ to $\{E_j\}$ and $\{E_j\}$ to $\{E_0\}$. Indeed, since $E_j\in X_j$ and is attractive on $\partial X_i\cap  X_j$ (Theorem \ref{Thm:Attractiveness_E1E2_1eq}), there is no complete orbit connecting $E_j$ to $E_0$ in $\partial X_i$.
        \end{itemize}
    \item[3) $\{E_0\}$ is isolated in $\partial X_i$.] Let $V$ be a neighborhood of $E_0$ in $\partial  X_i$ such that $E_j\notin V$ and let $K\subset V$ be a compact invariant set. Since $E_j$ is globally attractive in $\partial X_i\cap  X_j$ (Theorem \ref{Thm:Attractiveness_E1E2_1eq}.1) and $E_j\notin K$, then  $K\subset \cap_{k\in\{1,2\}} \partial X_k$. Since $K$ is compact and invariant, through any $u\in K$ there is a bounded and complete orbit $\gamma(u)\subset K$. Since the constant orbit $\{E_0\}$ is the only complete orbit included in  $\cap_{k\in\{1,2\}}\partial X_k$, it follows that $K=\{E_0\}$.
    \item[4) $\{E_j\}$ is isolated in $\partial X_i$.] Suppose $V\subset \partial X_i\cap  X_j$ is a neighborhood of $E_j$ and $K\subset V$ is a compact invariant set such that $E_j\in K$. Since $K$ is compact and invariant, for any $u\in K$ there is a bounded and complete orbit $\gamma(u)\subset K$ through $u$. Then from Lemma \ref{Lemma:Lyap_welldef_Orbit},  $L_j$ is a Lyapunov functional on $K$ and, for any $u\in K$,  $L_j$ is constant on $\gamma(u)$ if and only if $u=E_j$. Then \cite[Theorem 2.52]{SmithThieme2011} implies that $K=\{E_j\}$.  If $K'\subset V$ is a compact and invariant set such that $E_j\not\in K'$, one can repeat the previous argument with $K=K'\cup \{E_j\}$, then conclude that $K=\{E_j\}$, so $K'=\emptyset$.

    \item[5) $\{E_0\}$ is uniformly weakly repelling in $X_i$.] See the proof of Proposition \ref{Prop:Lyap_welldef_1eq}.  
    \item[6) $\{E_j\}$ is uniformly weakly repelling in $X_i$.] The proof is identical to the latter point (see Proposition \ref{Prop:Lyap_welldef_1eq}) after replacing $\RR_{0,i}$, $E_0$ and $x_0^*$ respectively by $\gamma_{j\to i}$, $E_j$ and $x_j^*$.
    \end{description}
    
    Then $\Phi$ is uniformly weakly  $\rho_i$-persistent  and the existence of a global attractor (corollary \ref{Cor:attractor}) implies that it is actually uniformly $\rho_i$-persistent (\cite[Theorem 5.2 and hypothesis H0 p.125]{SmithThieme2011}).

    \item This a direct consequence of the uniform $\rho_i$-persistence for $i\in\{1,2\}$ stated in item 3.

    \item We have $\rho_{1+2}^{-1}(\{0\})= \partial X_1 \cap \partial X_2$ and $\underset{u\in \partial X_1 \cap \partial X_2}{\bigcup} \omega(u)=\{E_0\}$ (by Theorem \ref{Thm:Attractiveness_E0}). Then we can prove,   exactly as  in the proof of Proposition \ref{Prop:Lyap_welldef_1eq}, than $\{E_0\}$ is compact, invariant, isolated and  acyclic in $\partial X_1 \cap \partial X_2$ and uniformly weakly repelling in $X_1 \cup X_2$. We conclude that $\Phi$ is uniformly weakly $\rho_{1+2}$-persistence, and the existence of a global attractor (corollary \ref{Cor:attractor}) implies that it is actually uniformly $\rho_{1+2}$-persistent (\cite[Theorem 5.2 and hypothesis H0 p.125]{SmithThieme2011}).
\end{enumerate}
\end{proof}

We can now compute the basin of attraction of each equilibrium, by means of the LaSalle invariance principle  \cite[Theorem 2.52]{MiSmiThi95}.

\begin{proof}[Proof of Theorem \ref{Thm:AttractivenessR0>1}]
    The proof proceeds exactly as for Theorem \ref{Thm:Attractiveness_E0} (with $L_*=L_i$ and $E_*=E_i$ for item 1 and $L_*=L_3$ and $E_*=E_3$ for item 2), except that here $S$ is not closed. In this case, the inclusion $\omega(K)\subset S$  is obtained because $S$ is invariant and $\Phi$ is uniformly $\rho_i$-persistent for item 1 or $\rho_{1,2}$-persistant for item 2 (Proposition \ref{Prop:Lyap_welldef}, items 3 and 4 respectively).  
\end{proof}

The critical case $\gamma_{1\to 2}=\gamma_{2\to 1}=1$ is handled by Theorem \ref{Thm:Attractiveness_continuum}.

\begin{proof}[Proof of Theorem \ref{Thm:Attractiveness_continuum}]
The proof proceeds in three steps. First we prove that $E_\infty\cup \{E_1,E_2\}$ is a compact attractor of compact sets in $\cup_{k\in\{1,2\}} X_k$ and is stable. Then we prove thate $E_\infty$ is an attractor of points in $\X_1\cap X_2$. Finally we prove the formula limit of $\Phi_t(v)$.

\begin{enumerate}
    \item Let $K\subset X_1\cup X_2$ be a compact set. Exactly as in the proof of Theorem \ref{Thm:Attractiveness_E0}, we know that $K$ is attracted by a compact invariant set $\omega(K)$. It comes from the uniform $\rho_{1+2}$-persistence (Proposition \ref{Prop:Lyap_welldef}) that $\omega(K)\subset X_1\cup X_2$ and from Lemma \ref{Lemma:Lyap_welldef_Orbit} that for any $i\in\{1,2\}$, $L_i$ is a Lyapunov functional on $\omega(K)\cap X_i$ (if it is not empty), strict for $E=E_\infty\cup(\{E_i\} \cap \omega(K))$. 

Let $v\in \omega(K)$. Since $\omega(K)$ is invariant and compact, there exists a bounded and complete orbit $\gamma(v)=\{\Psi(t),t\in \R\}\subset \omega(K)$ through $v$. The compactness of $\omega(K)$ implies that $\omega(v)\subset \omega(K)$ and $\alpha(v)\subset \omega(K)$.

If $v\in X_i \cap \partial X_j$ for $i,j\in\{1,2\},\ i\neq j$, then the invariance of $\partial X_j$ and the fact that $\omega(v)\cup \alpha(v) \subset \omega(K)\subset X_1\cup X_2$ implies that $\omega(v)\cup \alpha(v)\subset X_i \cap \partial X_j \cap \omega(K)$. Then we conclude as in the proof of Theorem \ref{Thm:Attractiveness_E0} that $v=\{E_i\}$.  In particular, $\omega(K)\cap X_i \cap \partial X_j$ is either empty, or reduces to $\{E_i\}$.

If $v\in X_1 \cap X_2$, then both $L_1$ and $L_2$ are well defined Lyapunov function on $\gamma(v)$ (Lemma \ref{Lemma:Lyap_welldef_Orbit}). Moreover, since $\alpha(v)\subset \omega(K)$, the previous paragraph implies that either $\alpha(v)\subset X_1 \cap X_2$, or there exists $i\in \{1,2\}$ such that $E_i\in\alpha(v)$, and the same applies for $\omega(v)$.

Suppose that $E_i\in \alpha(v)$, \textit{i.e.} $\exists (t_n)_{n\in\N}\in \R_+^\N$ such that $\lim_{n\to +\infty} t_n=+\infty$ and $\lim_{n \to +\infty}\Psi(-t_n)=E_i$. We know that $L_i(E_i)=0$ and $t\longmapsto L_i(\Psi(t))$ is a non-increasing non-negative function, then it comes that $L_i(\Phi(t))=0$ for all $t\in \R$, so $\gamma(v)=\{E_i\}$. In particular $v=E_i$, which is absurd. Then $E_i\not\in \alpha(v)$. Similarly, $E_i\not \in \omega(v)$ because $t\longmapsto L_j(\Psi(t))$ (with $j\in \{1,2\}$, $j\neq i$) is a non-increasing function and $\lim_{\Psi(t)\to  E_i} L_j(\Psi(t))=+\infty$. We just proved that $\alpha(v)\cup \omega(v) \subset X_1 \cap X_2$.

We know that $L_1$ and $L_2$ are constant on the invariant sets $\alpha(v)$ and $\omega(v)$ (\cite[Proposition 2.51]{SmithThieme2011}). In particular, for any $i\in\{1,2\}$ and any $u\in \alpha(v) \cup \omega(v)$, $t\longmapsto L_i\circ \Phi_t(u)$ is constant. Then Lemma \ref{Lemma:Lyap_welldef_Orbit}  implies that $u\in E_\infty$, so that $\omega(v)\subset E_\infty$ and $\alpha(v)\subset E_\infty$.  

Now we prove that $\omega(v)$ and $\alpha(v)$ are singletons. Indeed, Proposition \ref{Prop:Positive_equilibria} states that $x_1^*=x_2^*$ and that $E_\infty=\{E_{1+\sigma}:= (1-\sigma)E_1 +\sigma E_2 \ : \ \sigma\in(0,1)\}$. Computing the values of $L_1$ and $L_2$ alongside this parametrization leads to
$$L_1(E_{1+\sigma})=0+\frac{y_1^*}{c_1} g(1-\sigma) +  \sigma \frac{y_2^*}{c_2}, \quad L_2(E_{1+\sigma})=0+\frac{y_2^*}{c_2}g(\sigma)  + (1-\sigma)\frac{y_1^*}{c_1},\qquad \forall \sigma\in(0,1).$$

Since $g$ is a decreasing function on $(0,1)$, then $\sigma \longmapsto L_1(E_{1+\sigma})$ and $\sigma \longmapsto L_1(E_{1+\sigma})$ are respectively increasing and decreasing functions on $(0,1)$. In particular, the only subsets of $E_\infty$ on which $L_1$ or $L_2$ are constant are its singletons. Since $L_1$ and $L_2$ are constant on the invariant sets $\alpha(v)$ and $\omega(v)$, then they both reduce to singletons, say $\alpha(v)={E_\alpha}$ and $\omega(w)={E_\omega}$. 

We prove that $E_\alpha=E_\omega$. Indeed, assume that $E_\alpha\neq E_\omega$, then the opposite monotonicities of $L_1$ and $L_2$ restricted to $E_\infty$ imply that either $L_1(E_\alpha)<L_1(E_\omega)$ or $L_2(E_\alpha)<L_2(E_\omega)$. On the other hand, for any $i\in\{1,2\}$, the Lyapunov functional $L_i$ is non-increasing along the complete orbit $\gamma(v)$, then $L_i(E_\alpha)=\lim_{t\to-\infty}L_i(\Psi(t)\geq \lim_{t\to+\infty}L_i(\Psi(t)=L_i(E_\omega)$. This is absurd, so $\alpha(v)=\omega(v)$ and reduces to a singleton $\{E_{1+\sigma_v}\}\subset E_\infty$. 

Then $\gamma(v)$ is a homoclinic loop connecting $E_{1+\sigma_v}$ to itself and we can conclude, as in the proof of Theorem \ref{Thm:Attractiveness_E0}, that $\gamma(v)$ reduces to $\omega(v)$, \textit{i.e.} $\{v\}=\omega(v)\subset  E_\infty$. 

We proved that $\omega(K)\subset E_\infty\cup \{E_1,E_2\}$, \textit{i.e.} $E_\infty\cup \{E_1,E_2\}$ is a compact attractor of the compact sets of $X_1\cup X_2$. Since $\Phi$ is also state-continuous, uniformly in finite time, \cite[Theorem 2.39]{SmithThieme2011} states that $ E_\infty\cup \{E_1,E_2\}$ is also stable.
 \item  Now we prove that $E_\infty$ is globally attractive in $\X_1\cap \X_2$.  Let $v=(x_0,y_{1,0},y_{2,0}) \in\X_1\cap X_2$, and denote, for all $t\geq 0$,  $\Phi_t(v)=(x(t,.),y_1(t),y_2(t))$.  The third item of Lemma \ref{Lemma: CNS_eq_inf} implies that $$\frac{y_1'(t)}{y_1(t)}=\frac{\mu_1 y_2'(t)}{\mu_2 y_2(t)}\text{, that is } (\ln \circ y_1)'(t)=\frac{\mu_1}{\mu_2}(\ln \circ y_2)'(t).$$

    After integration, it comes that $y_1$ and $y_2$ are linked by the following algebraic relation
    \begin{equation} y_1(t)=\frac{y_{1,0}}{y_{2,0}}y_2(t)^\frac{\mu_1}{\mu_2} \label{Eq:Relation_y1_y2}
    \end{equation}
    
Suppose that $\underset{t\to+\infty}{\liminf} y_1(t)=0$, then there exists $(t_n)_{n\in\mathbb{N}}$ such that $t_n\underset{n\to +\infty}{\to}+\infty$ and $y_1(t_n)\underset{n\to +\infty}{\to}0$. Then, from equation \eqref{Eq:Relation_y1_y2}, $y_2(t_n)\underset{n\to +\infty}{\to}0$. From Lemma \ref{Lemma:omega-limit},  $\omega(v)$ is a non empty attractor of $v$, then there exists $\tilde{x}\in L^1(\R_+,\R_+)$ such that $(\tilde{x},0,0)\in \omega(v)$, otherwise the sequence $(x(t_n),y_1(t_n),y_2(t_n))$ would not belong to any neighborhood of $\omega(v)$. This is absurd as we proved that $E_\infty\cup\{E_1,E_2\}$ is an attractor of the compact sets of $X_1\cup X_2$,  in particular $\omega(\{v\})\subset E_\infty\cup\{E_1,E_2\}$, while $(\tilde{x},0,0)\not \in E_\infty\cup\{E_1,E_2\}$ for any $\tilde{x}\in L^1(\R_+,\R_+)$.

Then we proved that $\omega(v)\subset E_\infty$, so $E_\infty$ is an attractor of points in $X_1\cap X_2$ and its stability is directly inherited from the stability of $E_\infty \cup \{E_1,E_2\}$ in $X_1 \cup X_2$.

\item We proved that $E_\sigma:=\omega(v)\subset E_\infty$. Let $\sigma \in (0,1)$ such that $E_{1+\sigma}= (x_1^*,(1-\sigma)y_1^*, \sigma y_2^*)\in \omega(v)$. Then there exists $(t_n)_{n\in\N}\in\R_+^\N$ such that $\lim_{n\to +\infty}t_n=+\infty$ and $\lim_{n\to +\infty}(y_1(t_n),y_2(t_n))=((1-\sigma)y_1^*, \sigma y_2^*)$. Consequently, equation \eqref{Eq:Relation_y1_y2} implies that \begin{equation}\label{Eq:sigma}
(1-\sigma)y_1^*=\dfrac{y_{1,0}}{y_{2,0}}(\sigma y_2^*)^{\mu_1/\mu_2}.
\end{equation}

This equation admits at least one solution since $\omega(v)\neq \emptyset$, and at most one because its left-hand side is a decreasing function of $\sigma$ while its right-hand side is an increasing one.
\end{enumerate}
\end{proof}

\begin{remark}
We could actually note that the set $E_\infty$ is the attractor of compact sets in $X_1\cap X_2$. This comes from the fact that for any compact set $K\subset X_1\cap X_2$, then the function $K\ni v:=(x_0,y_{1,0},y_{2,0})\longmapsto \sigma\in (0,1)$ is continuous, where $\sigma$ is the unique solution of \eqref{Eq:sigma}. Consequently this function attains its bounds $0<\sigma_1<\sigma_2<1$ and $\omega(K)\subset \{(x,y_1,y_2)\in X: y_1\geq (1-\sigma_1)y_1^*, y_2\geq \sigma_1 y_2^*\}\subset X_1\cap X_2$. The first item of the latter proof can then be used similarly.
\end{remark}

\section{Discussion}
In this work, we have shown that introducing age-structure in a resource population shared by two competitors can fundamentally alter the outcome of competition. In contrast with models where only the competitors are structured, for which competitive exclusion generically prevails, we prove that heterogeneity in the interaction kernels may allow for stable coexistence.  We give necessary and sufficient conditions for existence, attractivity and stability of equilibria, including in the critical case where a continuum of infinite number of equilibria exists.

From a biological perspective, our results support the idea that specialization with respect to resource age can act as a mechanism of niche partitioning. The condition that the interaction kernels should not be proportional for stable coexistence (Remark \ref{Remark:ProportionalBeta}) reflects the need for \textit{sufficiently} distinct competitive strategies.

This study could be generalized to $n>2$ competitors, with most of the proofs being barely affected. The major difficulty would lie in determining the conditions for the existence of equilibria with more that two competitors.

As pointed out in Remark \ref{Remark:OptimalBeta}, considering a constant influx of resources of age $0$ implies that the best competitive strategy is to focus on the youngest resources. This was illustrated in Example \ref{Exp:analytic}, where the two predators differ only in their interaction term by a translation factor.

However, feeding partly on the youngest resources is not enough to ensure survival. As shown numerically in Example \ref{Exp:Bifurcation}, a generalist population $y_2$ feeding on resources of all ages may not survive against a specialist population $y_1$ that feeds (exlusively) on resources within a specific age range. From an epidemiological perspective, this can be interpreted as different \textit{Plasmodium} species competing for red blood cells, where \textit{P. vivax} and \textit{P. ovale} preferentially invade reticulocytes (young red blood cells), \textit{P. malariae} primarily targets mature cells, whereas \textit{P. falciparum} can infect cells across all age classes \cite{PaulEtAl2003}.

 This radical optimal competitive strategy could be altered when considering two alternative assumptions that we are exploring in ongoing works. First one can consider the case of linear birth rate (\textit{i.e.} $x(t,0)=\int_0^\infty \lambda(a)x(t,a)\d a$), so that the competitors need to preserve the sustainability of the resource. Alternatively, in an ecological setting, one could reasonably hypothesize that the nutritive value of a prey depends on its age (\textit{i.e.} non-constant functions $c_1$ and $c_2$), this would change the picture in a non trivial way, as predators would face contradictory interests :  being the first to reach the prey while waiting for it to grow before hunting it.  
\bibliographystyle{abbrv}
\bibliography{Biblio.bib}
\end{document}